\documentclass[a4paper,reqno]{amsart}

\usepackage{amsmath}
\usepackage{paralist}
\usepackage{graphics}
\usepackage{epsfig}
\usepackage{amsfonts}
\usepackage{amssymb}
\usepackage{hyperref}
\usepackage{epstopdf}
\usepackage{tikz}
\usetikzlibrary{intersections}
\usepackage{xcolor}
\usepackage{color}

\usepackage{mathpazo}

\tikzset{elegant/.style={smooth,thick,samples=50,cyan}}
\tikzset{eaxis/.style={->,>=stealth}}

\newtheorem{theorem}{Theorem}[section]
\newtheorem{lemma}[theorem]{Lemma}

\def\x#1{(\ref{#1})}

\def\ifl{\iffalse }

\def\bc{\begin{center}}       \def\ec{\end{center}}
\def\ba{\begin{array}}        \def\ea{\end{array}}
\def\be{\begin{equation}}     \def\ee{\end{equation}}
\def\bea{\begin{eqnarray}}    \def\eea{\end{eqnarray}}
\def\beaa{\begin{eqnarray*}}  \def\eeaa{\end{eqnarray*}}

\numberwithin{equation}{section}

\newtheorem{definition}[theorem]{Definition}

\newtheorem{remark}[theorem]{Remark}
\numberwithin{equation}{section}

\begin{document}
	
	\title[The critical lines for a chemotaxis  system]
	{Sharp conditions on global existence and blow-up in a degenerate two-species and cross-attraction system}

	\author{Jos\'{e} Antonio Carrillo}
	\address{Mathematical Institute, University of Oxford, Oxford, OX2 6GG, UK}
	\email{carrillo@maths.ox.ac.uk}
	
	\author{Ke Lin}
	\address{School of Economics and Mathematics, Southwestern University of Economics and Finance, Chengdu, 610074 SICHU China}
	\email{linke@swufe.edu.cn}

	

	
	\begin{abstract}
			We consider a degenerate chemotaxis model with two-species and two-stimuli in dimension $d\geq 3$ and find two critical curves intersecting at one same point which separate the global existence and blow up of weak solutions to the problem. More precisely, above these curves (i.e. subcritical case), the problem admits a global weak solution obtained by the limits of strong solutions to an approximated system. Based on the second moment of solutions, initial data are constructed to make sure blow up occurs in finite time below these curves (i.e. critical and supercritical cases). In addition,  the existence or non-existence of minimizers of free energy functional is discussed on the critical curves and the solutions  exist globally in time if the size of initial data is small. We also investigate the crossing point between the critical lines in which a refined criteria in terms of the masses is given again to distinguish the dichotomy between global existence and blow up. We also show that the blow ups is simultaneous for both species.

	\end{abstract}

	\maketitle
	
	\section{Introduction}	
	The interaction motion of two cell populations in breast cancer cell invasion models in $\mathbb{R}^d$ ($d\geq 3$) have been described by the following chemotaxis system with two chemicals and nonlinear diffusion (cf. \cite{EVC19,KPE14-JTB})
	\be \label{TSTC}\begin{cases}
		u_t =  \Delta u^{m_1}-\nabla \cdot (u\nabla v), &x\in \mathbb{R}^d, t>0, \\[0.2cm]
		-\Delta v = w,  & x\in \mathbb{R}^d, t>0, \\[0.2cm]
		w_t =  \Delta w^{m_2}-\nabla \cdot (w\nabla z), &x\in \mathbb{R}^d, t>0, \\[0.2cm]
		-\Delta z = u,  & x\in \mathbb{R}^d, t>0, \\[0.2cm]
		u(x,0)=u_0(x), \,\,\,w(x,0)=w_0(x),& x\in \mathbb{R}^d,
	\end{cases}
	\ee
	where $m_1,m_2>1$ are constants.  Here, $u(x,t)$ and $w(x,t)$ denote the density of the macrophages and the tumor cells, $v(x,t)$ and $z(x,t)$ denote the concentration of the chemicals produced by $w(x,t)$ and $u(x,t)$, respectively. For simplicity, the initial data are assumed to satisfy
	\be{\label{intial data for u and w}}
	\begin{split}
		u_0&\in L^1(\mathbb{R}^d;(1+|x|^2)dx)\cap L^{\infty}(\mathbb{R}^d),\,\,\,\nabla u^{m_1}_0\in L^2(\mathbb{R}^d)\,\,\,\text{and}\,\,\,u_0\geq 0,\\
		w_0&\in L^1(\mathbb{R}^d;(1+|x|^2)dx)\cap L^{\infty}(\mathbb{R}^d),\,\,\,\nabla w^{m_2}_0\in L^2(\mathbb{R}^d)\,\,\,\text{and}\,\,\,w_0\geq 0.\\
	\end{split}
	\ee
	Since the solutions to the Poisson equations can be written by the Newtonian potential such as
	\begin{equation*}
		\begin{split}
			v(x,t)=&\mathcal{K}\ast w=c_d\int_{\mathbb{R}^d}\frac{w(y,t)}{|x-y|^{d-2}}dy, \quad z(x,t)=\mathcal{K}\ast u=c_d\int_{\mathbb{R}^d}\frac{u(y,t)}{|x-y|^{d-2}}dy
		\end{split}
	\end{equation*}
	with $\mathcal{K}(x)=\frac{c_d}{|x|^{d-2}}$ and $c_d$ is the surface area of the sphere $\mathbb{S}^{d-1}$ in $\mathbb{R}^d$, the original system (\ref{TSTC}) can be regarded as the interaction between two populations
	\be \label{Simple TSTC}\begin{cases}
		u_t =  \Delta u^{m_1}-\nabla \cdot (u\nabla \mathcal{K}\ast w), &x\in \mathbb{R}^d, t>0,  \\[0.2cm]
		w_t =  \Delta w^{m_2}-\nabla \cdot (w\nabla \mathcal{K}\ast u), &x\in \mathbb{R}^d, t>0, \\[0.2cm]
		u(x,0)=u_0(x), \,\,\,w(x,0)=w_0(x),& x\in \mathbb{R}^d,
	\end{cases}
	\ee
	where it follows that the solutions obey the mass conservation
	\begin{equation*}
		\begin{split}
			M_1:=&\int_{\mathbb{R}^d}u(x,t)dx=\int_{\mathbb{R}^d}u_0(x)dx\,\,\,\text{and}\,\,\,
			M_2:=\int_{\mathbb{R}^d}w(x,t)dx=\int_{\mathbb{R}^d}w_0(x)dx.
		\end{split}
	\end{equation*}
	The associated free energy functional $\mathcal{F}$ for (\ref{TSTC}) or (\ref{Simple TSTC}) is given by
	\begin{align*}
		\mathcal{F}[u(t),w(t)]
		=&\frac{1}{m_1-1}\int_ {\mathbb{R}^d}u^{m_1}dx+\frac{1}{m_2-1}\int_ {\mathbb{R}^d}w^{m_2}dx-c_d\mathcal{H}[u,w],
	\end{align*}
	which is non-increasing with respect to time since for smooth case it satisfies the following decreasing property
	\begin{align*}
		\frac{d}{dt}\mathcal{F}[u(t),w(t)]=&-\int_ {\mathbb{R}^d}u\Big|\frac{m_1}{m_1-1}\nabla u^{m_1-1}-\nabla v\Big|^2dx\\
		&-\int_ {\mathbb{R}^d}w\Big|\frac{m_2}{m_2-1}\nabla w^{m_2-1}-\nabla z\Big|^2dx,
	\end{align*}
	where
	\begin{equation*}
		\begin{split}
			\mathcal{H}[u,w]=\iint_ {\mathbb{R}^d\times \mathbb{R}^d}\frac{u(x,t)w(y,t)}{|x-y|^{d-2}}dxdy.
		\end{split}
	\end{equation*}
		Only one-single population and chemical signal consisting of chemotaxis system is the well-known Keller-Segel model by taking into account volume filling constraints (see \cite{KS70-JTB,PH1996-CAM,CC2006-JMP}) reading as
		\be \label{Keller Segel}
		\begin{cases}
			u_t =  \Delta u^{m_1}-\nabla \cdot (u\nabla \mathcal{K}\ast u), &x\in \mathbb{R}^d, t>0,  \\[0.2cm]
			u(x,0)=u_0(x), & x\in \mathbb{R}^d,
		\end{cases}
		\ee
		which has immensely investigated over the last decades. See \cite{BBTW15-MMM,Horstmann2003-DMV,KS70-JTB,Perthame07,CCY} for the biological motivations and a complete overview of mathematical results for related more general aggregation-diffusion models. Here the diffusion exponent $m_1$ is taken to be supercritical $0<m_1<m_c:=2-2/d$, critical $m_1=m_c$ and subcritical $m_1>m_c$  if $d\geq 3$. The critical number $m_c$ is chosen to produces a balance between diffusion and potential drift in mass invariant scaling. For the subcritical $m_1>m_c$ in the sense that diffusion dominates, the solutions are globally solvable without any restriction on the size of the initial data \cite{KY12-SJMA,Sugiyama06-DIE,Sugiyama06-JDE}. However, in the supercritical case, the attraction is stronger leading to a coexistence of global existence of solutions and blow-up behavior. More precisely, finite-time blow up occurs for large initial data, see \cite{CCE12-CPDE} for $m_1=1$, \cite{CLW12} for $m_1=2d/(d+2)$, \cite{CW14-DM} for $2d/(d+2)<m_1< m_c$, and \cite{Sugiyama06-DIE} for $1<m_1<m_c$. But there also exists a global weak solution with decay properties under some smallness condition on the initial mass \cite{BianLiu2013,CLW12, CPZ04-MJ,Sugiyama06-JDE}. The critical case $m_1=m_c$ is investigated in \cite{Carrillo09-CVPDE,Sugiyama07-ADE} showing the existence of a sharp mass constant $M^*$ allowing for a dichotomy: if $\|u\|_1=M_1<M^*$ the solutions exist for all time, whereas if $M_1\geq M^*$ there exists solution with non-positive free energy functional blowing up. In addition, such similar dichotomy  was found in \cite{BDP2006-BDP, DP2004-CRMASP,JL1992-TAMS}  earlier in dimension $d=2$ and linear diffusion $m_1=1$ for (\ref{Keller Segel}) with $\mathcal{K}(x)=-1/(2\pi)\log |x|$ , where $M^*$ was replaced by $8\pi$. We also note that the results in \cite{BCM2008-CPAM} prove that solutions blow up as a delta Dirac at the center	of mass as time increases in critical mass $M_1=8\pi$. Sufficient conditions for nonlinear diffusion $m_1>1$ to prevent blow up are derived in \cite{CC2006-JMP}.
		
		The variational viewpoint to analyse problems of the type (\ref{Keller Segel}) has also been an active field of research. For instance, there have been recent results about the properties of global minimizers of the corresponding free energy functional, including the existence, radial symmetry and uniqueness and so on, since they  not only correspond to steady states of  (\ref{Keller Segel}) in some particular cases, but also are candidates for the large time asymptotics of solutions to (\ref{Keller Segel}).  Lion's concentration-compactness principle  \cite{Lion1984-Poincare} (see also \cite{Bedrossian11-AML}) can be directly  applied to the subcritical $m_1>m_c$ if $d\geq 3$ and allows the existence of minimizer which further satisfies some regularities properties (see \cite{CHMV18-CVPDE}). The uniqueness of minimizer in this case is ensured in \cite{LY1987-CMP} and such unique minimizer is also an exponential attractor of solutions of (\ref{Keller Segel}) when the initial data is radially symmetric and compactly supported by using the mass comparison principle (see \cite{KY12-SJMA}). In the critical case $m_1=m_c$, the free energy functional doses not admit global minimizers except for the critical mass case $M_1=M^*$ introduced above \cite{CCH07}. Such minimizers were used in \cite{Carrillo09-CVPDE} to describe the infinite time blow-up profile. For the nonlinear-diffusion in two dimension, the long time asymptotics of solutions is fully characterized in \cite{CHV19-IM} based on the  unique existence of radial minimizer of $\mathcal{F}$ \cite{CCV15-SJMA}. We  refer to \cite{BCC12-JFA} for a discussion on the existence of many stationary states for $m_1=1$ and $d=2$ in the critical case $M_1=8\pi$ and their basins of attraction.
		
		Back to linear two-species system (\ref{TSTC}) in $d=2$, similar to the role of the critical mass $8\pi$ in (\ref{Keller Segel}) (\cite{BDP2006-BDP,  DP2004-CRMASP}),  the critical curve $M_1M_2-4\pi(M_1+M_2)=0$ for two species is discovered in \cite{HWYZ2019-Nonlinearity}: solutions exist globally if $M_1M_2-4\pi(M_1+M_2)<0$ and blow up occurs if $M_1M_2-4\pi(M_1+M_2)>0$. The key tool for the  proof of the global existence part is using the Moser-Trudinger inequality as in \cite{SW05-JEMS} in two dimensions. One can use partial results in \cite{SW05-JEMS} to check that mimimizers indeed exist in the case $M_1M_2-4\pi(M_1+M_2)=0$. We also mention that such nonlinear system (\ref{TSTC}) and the one-single population system (\ref{Keller Segel}) can be formally regarded as gradient flows of the free energy functional in the probability measure space with the Euclidean Wasserstein metric \cite{AGG2005,JKO1998-SJMA}. For general $n$-component multi-populations chemotaxis system, in \cite{KW2019-JDE,KW2020-EJAM} the authors have made considerable progress on these aspects and obtain the global arguments in subcritical and critical cases. The Neumann initial-boundary value problem is analysed in \cite{KX2019-NS, KX2020-CVPDE, TW2015-DCDS,YWZ2018-Non}.
		
		The aim  of this paper is to give a thorough understanding of the well-posedness and asymptotic behavior for (\ref{TSTC}) and (\ref{Simple TSTC}) in $d\geq 3$ and to show the existence or non-existence of global minimizers in critical cases. We make use of bold faces  $\bf{m},\bf{A},\bf{B},\bf{I},\bf{M},\cdots$ to denote two-dimensional vectors through the paper and assume that ${\bf{A}}=(a_1,a_2)\leq(\geq){\bf{B}}=(b_1,b_2)$ means that $a_1\leq(\geq)b_1$ and $b_1\leq(\geq)b_2$, respectively. If $(u,w)$ is a solution of (\ref{Simple TSTC}), then for any $\lambda>0$ the following scaling
		$$
		u_\lambda(x,t)=\lambda^{m_2}u(\lambda^{\frac{m_1+m_2-m_1m_2}{2}}x,\lambda^{m_1}t),\,\,
		w_\lambda(x,t)=\lambda^{m_1}w(\lambda^{\frac{m_1+m_2-m_1m_2}{2}}x,\lambda^{m_2}t)
		$$
		is also a solution, where the above scaling becomes mass invariant for both $u$ and $w$  if and only if ${\bf{m}}:=(m_1,m_2)=(m_c,m_c)$.
		When ${\bf{m}}$ satisfy
		\begin{equation}\label{critical line1}
			\begin{split}
				m_1m_2+2m_1/d=m_1+m_2,
			\end{split}
		\end{equation}
		the mass conservation law only holds for $w$, whereas only $u$ preserves $L^1$-norm if
		\begin{equation}\label{critical line2}
			\begin{split}
				m_1m_2+2m_2/d=m_1+m_2.
			\end{split}
		\end{equation}
		The curves (\ref{critical line1}) and (\ref{critical line2}) can be shown to be the sharp conditions separating the global existence and blow up. Our main result in Theorem 1.3 shows the following dichotomy: above the two red curves in Figure \ref{figregimes}, in the sense that $m_1m_2+2m_1/d> m_1+m_2$ or $m_1m_2+2m_2/d> m_1+m_2$, weak solutions globally exists and blow up occurs below the red curves for certain initial data regardless of their initial masses (see Theorem \ref{subcritical or supercritical theorem}). Several results are also obtained at the critical curves (see Theorem \ref{critical theorem}). In addition, both two lines will intersect at the point $(m_c,m_c)$. Therefore, we consider the $(m_1,m_2)\in (1,\infty)^2$ parameter range divided by the following three critical cases (red curve in Figure \ref{figregimes}):
		\begin{equation*}
			\begin{split}
				&\text{Line}\,\,\,L_1:\,\,\,m_1m_2+2m_1/d= m_1+m_2\,\,\,\text{with}\,\,\, m_1\in\left(m_c,d/2\right),\,\,\,m_2\in\left(1,m_c\right);\\
				&\text{Line}\,\,\,L_2:\,\,\, m_1m_2+2m_2/d= m_1+m_2\,\,\,\text{with}\,\,\, m_1\in\left(1,m_c\right),\,\,\,m_2\in\left(m_c,d/2\right);\\
				&\text{The intersection point}\,\,\,\text{{\bf I}}:=(m_c,m_c),
			\end{split}	
		\end{equation*}
		
		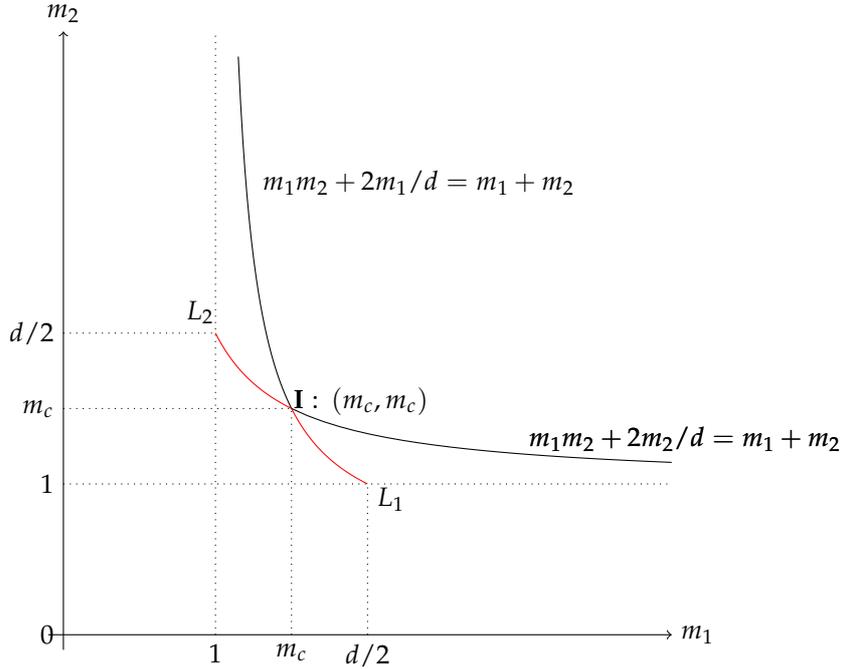
\begin{figure}[ht!]
			\label{figregimes}
			\centering
			\begin{tikzpicture}
				\draw (0,0) node[left] {$0$};
				\draw[->] (-0.2,0) --(8,0) node[right] {$m_1$};
				\draw[->] (0,-0.2) --(0,8) node[above] {$m_2$};
				\draw[dotted] (2,0)--(2,8);
				\draw[dotted] (0,2)--(8,2);
				\draw (2,0) node[below]{$1$};
				\draw (0,2) node[left]{$1$};
				\draw[dotted] (3,0)--(3,3) node[below] at(3,0) {$m_c$};
				\draw[dotted] (0,3)--(3,3) node[left] at(0,3) {$m_c$};
				\draw[dotted] (4,0)--(4,2) node[below] at(4,0) {$d/2$};	
				\draw[dotted] (0,4)--(2,4) node[left] at(0,4) {$d/2$};			
				\draw[domain =2.3:3] plot (\x ,{1+2/(\x-2)});
				\draw[domain =3:4,red] plot (\x ,{1+2/(\x-2)});
				\node[right] at(2.5,6) {$m_1m_2+2m_1/d=m_1+m_2$};
				\draw[domain =3:8] plot (\x ,{2+2/(\x-1)}) node[right] at(6,2.6)  {$m_1m_2+2m_2/d=m_1+m_2$};
				\draw[domain =2:3,red] plot (\x ,{2+2/(\x-1)});
				\node[right] at(6,2.6)  {$m_1m_2+2m_2/d=m_1+m_2$};
				\draw (2.9,3.1) node[right]{${\bf{I}}:\,(m_c,m_c)$};
				\node[right] at(4,1.8) {${L_1}$};
				\node[above] at(1.8,4) {${L_2}$};
			\end{tikzpicture}
			\caption{Parameter lines determining the critical regimes.}
		\end{figure}
		Based on the above discussion, we say that ${\bf{m}}=(m_1,m_2)$ is subcritical if
		\begin{equation*}
			m_1m_2+2m_1/d>m_1+m_2\,\,\,\,\text{or}\,\,\,\, m_1m_2+2m_2/d>m_1+m_2,
		\end{equation*}
		and ${\bf{m}}=(m_1,m_2)$ is supercritical if
		\begin{equation*}
			m_1m_2+2m_1/d<m_1+m_2\,\,\,\,\text{and}\,\,\,\, m_1m_2+2m_2/d<m_1+m_2.
		\end{equation*}
		Notice that this corresponds to be above (subcritical) or below (supercritical) the red curves in Figure \ref{figregimes}.
		We also define subsets of $L^1(\mathbb{R}^d)$ as
		\begin{equation*}\label{definition of S1}
			\begin{split}
				S_{M_1}:=\{f\geq 0:f\in L^1(\mathbb{R}^d)\cap L^{m_1}(\mathbb{R}^d)\,\,\,\text{and}\,\,\,\|f\|_1=M_1\}
			\end{split}
		\end{equation*}
		and
		\begin{equation*}\label{definition of S2}
			\begin{split}
				S_{M_2}:=\{g\geq 0:g\in L^1(\mathbb{R}^d)\cap L^{m_2}(\mathbb{R}^d)\,\,\,\text{and}\,\,\,\|g\|_1=M_2\}.
			\end{split}
		\end{equation*}
		Now the definition of weak solution for (\ref{TSTC}) or (\ref{Simple TSTC}) is give as
		\begin{definition}\label{weak solution}
			Let $m_1,m_2> 1$, $d\geq 3$ and $T>0$. Suppose the initial data $(u_0,w_0)$ satisfies some classical regularities (\ref{intial data for u and w}). Then $(u,w)$ of nonnegative functions defined in $\mathbb{R}^d\times(0,T)$ is called a weak solution if
			\begin{align*}
				i)&\,\,(u,w)\in (C([0,T);L^1(\mathbb{R}^d))\cap L^{\infty}(\mathbb{R}^d\times(0,T)))^2,\\
				&(u^{m_1},w^{m_2})\in (L^2(0,T;H^1(\mathbb{R}^d)))^2;\\
				ii)&\,\,(u,w)\,\,\text{satisfies}\\
				&\int^T_0\int_{\mathbb{R}^d}u
				\phi_{1t}dxdt+\int_{\mathbb{R}^d}u_0(x)\phi_1(x,0)dx=\int^T_0\int_{\mathbb{R}^d}(\nabla u^{m_1}-u\nabla v)\cdot \nabla \phi_1 dxdt,\\
				&\int^T_0\int_{\mathbb{R}^d}w
				\phi_{2t}dxdt+\int_{\mathbb{R}^d}w_0(x)\phi_2(x,0)dx=\int^T_0\int_{\mathbb{R}^d}(\nabla w^{m_2}-w\nabla z)\cdot \nabla \phi_2 dxdt,
			\end{align*}
			for any test functions $\phi_1\in \mathcal{D}(\mathbb{R}^d\times[0,T)) $ and $\phi_2\in \mathcal{D}(\mathbb{R}^d\times[0,T)) $ with $v=\mathcal{K}\ast w$ and $z=\mathcal{K}\ast u$.
			
		\end{definition}
		
		For a given weak solution, we also define:
		\begin{definition}\label{free energy solution}
			Let $T>0$. Then $(u,w)$ is called a free energy solution with some regular initial data $(u_0,w_0)$ on $(0,T)$ if $(u,w)$ is a weak solution and moreover satisfies
			$(u^{(2m_1-1)/2},w^{(2m_2-1)/2)})\in (L^2(0,T;H^1(\mathbb{R}^d)))^2$ and
			\be{\label{inequality for free energy2}}
			\begin{split}
				\mathcal{F}&[u(t),w(t)]+\int^t_ {0}\int_ {\mathbb{R}^d}u\Big|\frac{m_1}{m_1-1}\nabla u^{m_1-1}-\nabla v\Big|^2dxds\\
				&+\int^t_ {0}\int_ {\mathbb{R}^d}w\Big|\frac{m_2}{m_2-1}\nabla w^{m_2-1}-\nabla z\Big|^2dxds
				\leq \mathcal{F}[u_0,w_0]
			\end{split}
			\ee
			for all $t\in(0,T)$ with $v=\mathcal{K}\ast w$ and $z=\mathcal{K}\ast u$.
			
		\end{definition}
		
		Our first main result for (\ref{TSTC}) or (\ref{Simple TSTC}) above or below lines $L_1$ and $L_2$ is:
		\begin{theorem}\label{subcritical or supercritical theorem}
			Let  $m_1,m_2>1$. Suppose that the initial data $(u_0,w_0)$ with $\|u_0\|_1=M_1,\|w_0\|_1=M_2$ fulfills (\ref{intial data for u and w}). Then\\
			$i)$ If  ${\bf{m}}$ is subcritical, there exists a global free energy solution. \\
			$ii)$ If  ${\bf{m}}$ is supercritical, then one can construct large initial data ensuring blow up in finite time.
		\end{theorem}
		On the lines $L_1$, $L_2$ and intersection point $\bf{I}$, our second main result is as follows.
		\begin{theorem}\label{critical theorem}
			Let  $m_1,m_2>1$. Suppose that the initial data $(u_0,w_0)$ with $\|u_0\|_1=M_1,\|w_0\|_1=M_2$ fulfills (\ref{intial data for u and w}). Then
			
			\noindent $i)$ If  ${\bf{m}}$ is ${\bf{I}}$, then there exists a number $M_c>0$ such that if $M_1M_2<M^2_c$, solutions globally exist and if $M_1M_2/(M^{m_c}_1+M^{m_c}_2)>M^{2/d}_c/2$,  there exists a finite time blow-up solution. Moreover, non-zero global minimizers of $\mathcal{F}$ exist in $S_{M_1}\times S_{M_2}$ if we are at the crossing point ${\bf{M}}=(M_c,M_c)$.
			
			$ii)$ If  ${\bf{m}}$ is on $L_1$, there exists a number $M_{2c}>0$ with the following properties: if $M_2<M_{2c}$, solutions globally exist and $\inf_{f\in S_{M_1}}\inf_{g\in S_{M_2}}\mathcal{F}[f,g]=0$ if $M_2=M_{2c}$, but there exist no  non-zero global minimizers of $\mathcal{F}$ in $S_{M_1}\times S_{M_2}$. In addition, blow-up solution exists if
			\begin{equation*} \frac{\left(\int_{\mathbb{R}^d}u^{m_1/m_2}_0dx\right)^{m_2/m_1}\left(\int_{\mathbb{R}^d}w_0dx\right)}{\left(\int_{\mathbb{R}^d}u^{m_1/m_2}_0dx\right)^{m_2}+\left(\int_{\mathbb{R}^d}w_0dx\right)^{m_2}}>N_0\,\,\,\text{with some}\,\,\,N_0>0.
			\end{equation*}
			If ${\bf{m}}$ is on $L_2$, there exists $M_{1c}>0$ with the similar properties for $M_1$ and  blow-up solution exists if
			\begin{equation*}
				 \frac{\left(\int_{\mathbb{R}^d}u_0dx\right)\left(\int_{\mathbb{R}^d}w^{m_2/m_1}_0dx\right)^{m_1/m_2}}{\left(\int_{\mathbb{R}^d}u_0dx\right)^{m_1}+\left(\int_{\mathbb{R}^d}w^{m_2/m_1}_0dx\right)^{m_1}}>N_0.
			\end{equation*}	
			
			$iii)$ A simultaneous blow-up phenomenon exists if ${\bf{m}}$ is critical .
		\end{theorem}
		
		We summarize our second main result on the intersection point ${\bf{I}}$, see Figure \ref{figreg}. The blue curve $M_1M_2=M^2_c$ intersects with the green curve  $M_1M_2/(M^{m_c}_1+M^{m_c}_2)=M^{2/d}_c/2$ at the point ${\bf{J}}=(M_c,M_c)$. Theorem {\ref{critical theorem}} implies that below the curve $M_1M_2=M^2_c$ solutions globally exist and above the curve $M_1M_2/(M^{m_c}_1+M^{m_c}_2)=M^{2/d}_c/2$  blow up happens.

		\begin{figure}[ht!]
			\centering
			\begin{tikzpicture}
				\draw (0,0) node[left] {$0$};
				\draw[->] (-0.2,0) --(8,0) node[right] {$M_1$};
				\draw[->] (0,-0.2) --(0,8) node[above] {$M_2$};
				\draw[dotted] (3,0)--(3,3) node[below] at(3,0) {$M_c$};
				\draw[dotted] (0,3)--(3,3) node[left] at(0,3) {$M_c$};	
				\draw[domain =1.6:8, green] plot (\x ,{1+4/(\x-1)});
				\node[right] at(6,2) {$M_1M_2/(M^{m_c}_1+M^{m_c}_2)=M^{2/d}_c/2$};
				\draw[domain =1.2:7,blue] plot (\x ,{9/\x});
				
				\draw (2.9,3.1) node[right]{${\bf{J}}:\,(M_c,M_c)$};
				\node[right] at(6.5,1) {${M_1M_2=M^2_c}$};
			\end{tikzpicture}
			\caption{Parameter lines on intersection point $\bf{I}$.}
			\label{figreg}
		\end{figure}
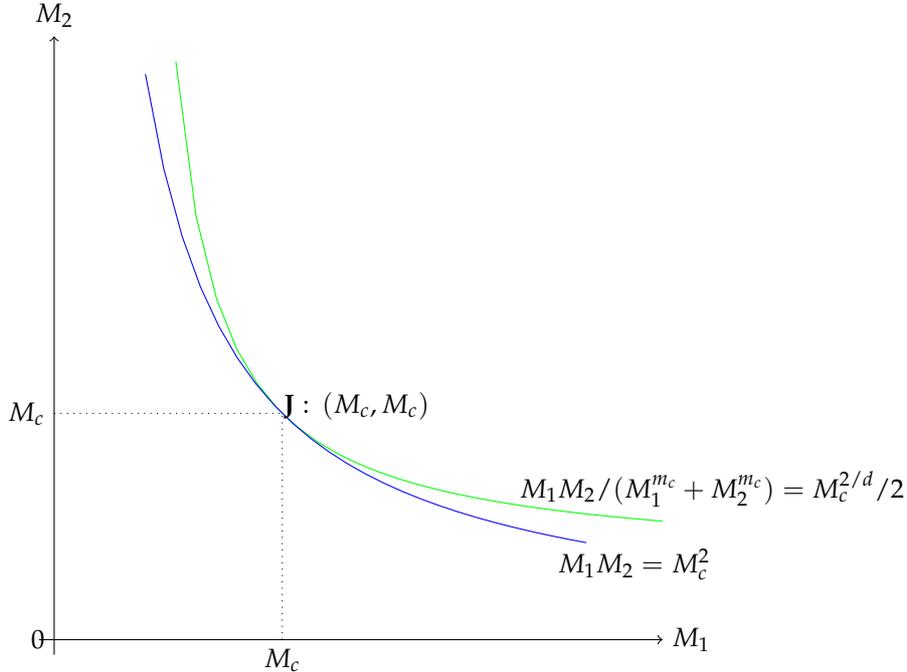
		It is an open problem to determine the sharp relation between the masses leading to dichotomy in the intersection point $\bf{I}$ and the long time asymptotics on the red curves $L_1$ and $L_2$ in Figure \ref{figregimes}.
		
		The organization of the paper is as follows: we first construct an approximated system for (\ref{TSTC}) in Section 2, and provide an sufficient condition for global existence of smooth solution and then obtain global weak solution or free energy solution of (\ref{TSTC}) by passing limits upon a prior estimate. Section 3 deals with properties of free energy functional, including the lower and upper bounds, and the existence or non-existence of non-zero minimizers if ${\bf{m}}$ is critical. Finally, we prove that the solutions are globally solvable if ${\bf{m}}$ is subcritical or critical with small initial data in Section 4 and construct blow-up solutions if  ${\bf{m}}$ is supercritical or critical with large masses in Section 5.
		
		\section{Approximated system}
		As mentioned in the introduction, we first consider an  approximated system
		\be \label{TSTC approxiamation}\begin{cases}
			u_{\epsilon t}(x,t) =  \Delta (u_\epsilon+\epsilon)^{m_1}-\nabla \cdot (u_{\epsilon}\nabla v_{\epsilon}), &x\in \mathbb{R}^d, t>0, \\[0.2cm]
			v_{\epsilon} = \mathcal{K}\ast w_\epsilon,  & x\in \mathbb{R}^d, t>0, \\[0.2cm]
			w_{\epsilon t}(x,t) =  \Delta (w_\epsilon+\epsilon)^{m_2}-\nabla \cdot (w_{\epsilon} \nabla z_{\epsilon} ), &x\in \mathbb{R}^d, t>0, \\[0.2cm]
			z_{\epsilon} = \mathcal{K}\ast u_\epsilon,  & x\in \mathbb{R}^d, t>0, \\[0.2cm]
			u_{\epsilon}(x,0)=u^{\epsilon}_0(x)\geq 0,w_{\epsilon}(x,0)=w^{\epsilon}_0(x)\geq 0,& x\in \mathbb{R}^d
		\end{cases}
		\ee
		with $u^{\epsilon}_0$ and   $w^{\epsilon}_0$ being the convolution of $u_0$ and $w_0$ with a sequence of mollifiers and $\|u^{\epsilon}_0\|_1=\|u\|_1=M_1$ and $\|w^{\epsilon}_0\|_1=\|w\|_1=M_2$. Then the uniform a priori estimate for solutions to (\ref{TSTC approxiamation}) is given if $m_1$ and $m_2$ are suitably large, thus global weak solution or even free energy solution exists by letting $\epsilon $ tends to 0.
		
		By virtue of the local existence of strong solution for only one-single population chemotaxis system (see {\cite[Proposition 4.1]{Sugiyama06-DIE}}), one obtains:
		\begin{lemma}\label{Local existence for AS}Let $m_1,m_2> 1$. Then there exists $T^{\epsilon}_{\max}\in(0,\infty]$ denoting the maximal existence time such that (\ref{TSTC approxiamation}) has a unique nonnegative strong solution $(u_{\epsilon},w_{\epsilon})\in \left(\mathbb{W}^{2,1}_{p}(Q_T)\right)^2 $ with some $p>1$, where $Q_T=\mathbb{R}^d\times (0,T)$ with $T\in(0,T^{\epsilon}_{\max}]$ and
			$$\mathbb{W}^{2,1}_{p}(Q_T):=\{u\in L^p(0,T;W^{2,p}(\mathbb{R}^d))\cap W^{1,p}(0,T;L^{p}(\mathbb{R}^d))\}.$$
			Moreover, if  $T^{\epsilon}_{\max}<\infty$, then
			$$
			\lim_{t\rightarrow T^{\epsilon}_{\max}}\left[\|u_{\epsilon}(\cdot,t)\|_{\infty}+\|w_{\epsilon}(\cdot,t)\|_{\infty}\right]=\infty.
			$$
		\end{lemma}
		
		Now we recall the Hardy-Littlewood-Sobolev (HLS) inequality which we frequently use later (see \cite{Lieb83-CMP}  or {\cite[Chapter 4]{Lieb2001}}).
		\begin{lemma}\label{HLS inequality}
			Let  $0<\lambda<d$, and let the Riesz potential $I_{\lambda}(h)$ of a function $h$ defined by
			$$
			I_{\lambda}(h)(x)=\frac{1}{|x|^{d-\lambda}}\ast h=\int_{\mathbb{R}^d}\frac{h(y)}{|x-y|^{d-\lambda}}dy,\,\,\,\,\,\,x\in\mathbb{R}^d.
			$$	
			Then for $h\in L^{\kappa_1}(\mathbb{R}^d)$ and for $\kappa_1,\kappa_2>1$ with $\frac{1}{\kappa_2}=\frac{1}{\kappa_1}-\frac{\lambda}{d}$, then there exists a sharp constant $C_{\text{HLS}}=C_{HLS}(d,\lambda,\kappa_1)>0$ such that
			$$
			\|I_{\lambda}(h)\|_{\kappa_2}\leq C_{\text{HLS}}\|h\|_{\kappa_1}.
			$$
		\end{lemma}
		An equivalent form of the  HLS inequality can be stated that if
		\begin{equation*}
			\begin{split}
				\frac{1}{p}+\frac{1}{q}=1+\frac{\lambda}{d},
			\end{split}
		\end{equation*}
		and  $h_1\in L^{p}(\mathbb{R}^d),h_2\in L^{q}(\mathbb{R}^d)$ with $p,q>1$, then there exists a $C_{HLS}=C_{HLS}(d,\lambda,p)>0$ such that
		\begin{equation*}
			\begin{split}
				\left|\iint_{\mathbb{R}^d\times\mathbb{R}^d}\frac{h_1(x)h_2(y)}{|x-y|^{d-\lambda}}dxdy\right|\leq C_{HLS}\|h_1\|_{p}\|h_2\|_q.
			\end{split}
		\end{equation*}

		Inspired by \cite{TaoWinkler12-JDE}, the global solvability of (\ref{TSTC approxiamation}) can be achieved based on assumptions on the boundedness for $\|u_\epsilon\|_{m_1}$ and 	$\|w_\epsilon\|_{m_2}$ with some large $m_1$ and $m_2$.
		
		\begin{lemma}\label{the condition for global existence}
			Let $T\in(0,T^{\epsilon}_{\max}]$. Assume that ${\bf{m}}$ satisfies
			\be{\label{connection between m_1 and m_2}}
			\begin{split}	
				m_1m_2+2m_1m_2/d> m_1+m_2.
			\end{split}
			\ee
			Suppose that there exists a constant $C>0$ such that $(u_{\epsilon},w_{\epsilon})$ of (\ref{TSTC approxiamation}) with initial data $(u^{\epsilon}_0,w^{\epsilon}_0)$ being the convolution of  $(u_0,w_0)$ satisfies
			\be\label{Asssumption on prior estimate for uwepsilon}
			\|u_{\epsilon}(t)\|_{m_1}\leq C\,\,\,\,\text{and}\,\,\,\,\|w_{\epsilon}(t)\|_{m_2}\leq  C\,\,\,\,\text{for}\,\,\,\,t\in(0,T).
			\ee
			Then there exists a constant $C=C(d,m_1,m_2,u_{0\epsilon},w_{0\epsilon})>0$ such that
			\be\label{prior estimate for uwepsilon2}
			\|(u_{\epsilon}(t),w_{\epsilon}(t))\|_{r}\leq  C\,\,\,\text{for}\,\,\,r\in[1,\infty)\,\,\,\,\text{and}\,\,\,\,t\in(0,T)
			\ee
			and\label{prior estimate for vzepsilon1}
			\be\label{3.1}
			\|(v_{\epsilon}(t),z_{\epsilon}(t))\|_{r}+\|(\nabla v_{\epsilon}(t),\nabla z_{\epsilon}(t))\|_{r}\leq  C\,\,\,\text{for}\,\,\,r\in[1,\infty]\,\,\,\,\text{and}\,\,\,\,t\in(0,T).
			\ee
		\end{lemma}
		
		\begin{proof}
			We split the proof into three steps. \\
			
			\noindent{\bf{Step 1.}}  \emph{The choices of $p$ and $q$}. There exist $\bar{p}>1,\bar{q}>1,r_1>1$  and $r_2>1$ such that for some $p>\bar{p}$ and $q>\bar{q}$ one has
			\be{\label{choice of p and q 1}}
			p>\begin{cases}m_1+1,\,\,\,&\text{if}\,\,\,m_1\geq\frac{d}{2},m_2\geq\frac{d}{2},\\[0.2cm]
				 \max\left\{m_1+1,\frac{(m_1-1)(m_2-1)d}{d-2m_2},\frac{m_1(d-2)}{2m_2}\right\},\,\,\,&\text{if}\,\,\,m_1\geq\frac{d}{2},m_2<\frac{d}{2},\\[0.2cm]
				 \max\left\{m_1+1,\frac{dm^2_1+d-2m_1}{d-2m_1},\frac{m_1(d-2)}{2m_2}\right\},\,\,\,&\text{if}\,\,\,m_1<\frac{d}{2},m_2\geq\frac{d}{2},\\[0.2cm]
				 \max\left\{m_1+1,\frac{dm^2_1+d-2m_1}{d-2m_1},\frac{(m_1-1)(m_2-1)d}{d-2m_2},\frac{m_1(d-2)}{2m_2}\right\},\,\,\,&\text{if}\,\,\,m_1<\frac{d}{2},m_2<\frac{d}{2},\\[0.2cm]
			\end{cases}
			\ee
			
			\be{\label{choice of p and q 2}}
			\begin{split}
				\frac{1}{r_1}<&1-\frac{d-2}{(q+m_2-1)d},
			\end{split}
			\ee
			
			\be \label{choice of p and q 3}
			\frac{1}{r_1}>\max\left\{1-\frac{1}{m_2},
			\frac{d-2}{d}\cdot\frac{p}{p+m_1-1}\right\},
			\ee
			
			\be{\label{choice of p and q 4}}
			\begin{split}	
				\frac{1}{r_2}>&\frac{d-2}{d}\cdot\frac{1}{p+m_1-1},\\
			\end{split}
			\ee

			\be \label{choice of p and q 5}
			\frac{1}{r_2}<\min\left\{
			\frac{1}{m_1},1-\frac{d-2}{d}\cdot\frac{q}{q+m_2-1}\right\}
			\ee
			
			and
			\be{\label{choice of p and q 6}}
			\begin{split}	
				\frac{\frac{p}{m_1}-\frac{1}{r_1}}{1-\frac{d}{2}+\frac{(p+m_1-1)d}{2m_1}}+	 \frac{\frac{1}{m_2}-1+\frac{1}{r_1}}{1-\frac{d}{2}+\frac{(q+m_2-1)d}{2m_2}}<\frac{2}{d},
			\end{split}
			\ee
			as well as
			\be{\label{choice of p and q 7}}
			\begin{split}	
				\frac{\frac{1}{m_1}-\frac{1}{r_2}}{1-\frac{d}{2}+\frac{(p+m_1-1)d}{2m_1}}+	 \frac{\frac{q}{m_2}-1+\frac{1}{r_2}}{1-\frac{d}{2}+\frac{(q
						+m_2-1)d}{2m_2}}<\frac{2}{d}.
			\end{split}
			\ee
			
			Let us first pick $r_1>1$ and $r_2>1$ fulfilling
			\be{\label{assumption on theta and mu}}
			\begin{split}
				r_1<\min\left\{\frac{d}{d-2},\frac{m_2}{m_2-1}\right\}
			\end{split}
			\ee
			and
			\be{\label{assumption on theta and mu2}}
			\begin{split}
				r_2>m_1,
			\end{split}
			\ee
			and let
			\begin{equation}{\label{defintion of q}}
				\begin{split}	
					q:=\frac{m_2(p-1)}{m_1}+1.
				\end{split}
			\end{equation}
			In (\ref{defintion of q}), $p>m_1+1$ implies $q>m_2+1$. The assertions in  (\ref{choice of p and q 1})-(\ref{choice of p and q 2}) and (\ref{choice of p and q 4}) easily hold by  sufficiently large $p\geq \bar{p}$ with some $\bar{p}>1$ and $q\geq \bar{q}$ with some $\bar{q}>1$.
			
			To see the possible choice of $r_1$ satisfying  (\ref{choice of p and q 2})-(\ref{choice of p and q 3}), we first observe that $1-\frac{1}{m_2}\geq \frac{d-2}{d}\cdot\frac{p}{p+m_1-1}$ is true for any $p>1$ if $m_2\geq\frac{d}{2}$, and $\frac{1}{r_1}>1-\frac{1}{m_2}$ holds by (\ref{assumption on theta and mu}) as well as $1-\frac{1}{m_2}<1-\frac{d-2}{(q+m_2-1)d}$ for any $q>1$. Thus the asserted $r_1$ can be actually found. When $m_2<\frac{d}{2}$, one has $\frac{1}{r_1}>\frac{d-2}{d}\cdot\frac{p}{p+m_1-1}
			> 1-\frac{1}{m_2}$. The first inequality is guaranteed by (\ref{assumption on theta and mu}) and the second is due to
			\begin{equation*}
				\begin{split}	
					\frac{d-2}{d}\cdot \frac{p}{p+m_1-1}> 1-\frac{1}{m_2}\Longleftrightarrow&
					\left(\frac{1}{m_2}-\frac{2}{d}\right)p> \frac{(m_1-1)(m_2-1)}{m_2}\\
					\Longleftrightarrow& p> \frac{(m_1-1)(m_2-1)d}{d-2m_2}
				\end{split}
			\end{equation*}
			by (\ref{choice of p and q 1}) if $m_2<\frac{d}{2}$. Moreover, from (\ref{defintion of q}) and  (\ref{choice of p and q 1}),
			$\frac{d-2}{d}\cdot\frac{p}{p+m_1-1}<1-\frac{d-2}{(q+m_2-1)d}$.  Therefore, one can also choose $r_1>1$ satisfying (\ref{choice of p and q 2})-(\ref{choice of p and q 3}) in the case $m_2<\frac{d}{2}$.
			
			Similar to the choice of  $r_2$, if $m_1\geq \frac{d}{2}$ then it follows from (\ref{assumption on theta and mu2}) that $\frac{1}{r_2}<\frac{1}{m_1}\leq
			1-\frac{d-2}{d}\cdot\frac{q}{q+m_2-1}$, in which
			(\ref{choice of p and q 4})-(\ref{choice of p and q 5}) can be satisfied due to $\frac{d-2}{d}\cdot\frac{1}{p+m_1-1}<\frac{1}{m_1}$. If $m_1<\frac{d}{2}$, (\ref{choice of p and q 1}) implies $\frac{d-2}{d}\cdot\frac{1}{p+m_1-1}<1-\frac{d-2}{d}\cdot\frac{q}{q+m_2-1}<\frac{1}{m_1}$, and the assertion is true.
			
			Since (\ref{connection between m_1 and m_2}) ensures
			\begin{equation*}
				\begin{split}	
					m_1/m_2-m_1< 2m_1/d-1,
				\end{split}
			\end{equation*}
			then
			\begin{equation*}{\label{choice of p and q 8}}
				\begin{split}	
					\frac{\frac{p}{m_1}-\frac{1}{r_1}}{1-\frac{d}{2}+\frac{(p+m_1-1)d}{2m_1}}&+	 \frac{\frac{1}{m_2}-1+\frac{1}{r_1}}{1-\frac{d}{2}+\frac{(q+m_2-1)d}{2m_2}}\\
					&=\frac{\frac{p}{m_1}-\frac{1}{r_1}}{1+\frac{(p-1)d}{2m_1}}
					+\frac{\frac{1}{m_2}-1+\frac{1}{r_1}}{1+\frac{(q-1)d}{2m_2}}\\
					&=\frac{\frac{p}{m_1}-\frac{1}{r_1}}{1+\frac{(p-1)d}{2m_1}}
					+\frac{\frac{1}{m_2}-1+\frac{1}{r_1}}{1+\frac{(p-1)d}{2m_1}}\\
					&=\frac{p+\frac{m_1}{m_2}-m_1}{p+\frac{2m_1}{d}-1}\cdot\frac{2}{d} <\frac{2}{d},
				\end{split}
			\end{equation*}
			and
			\begin{equation*}{\label{choice of p and q 9}}
				\begin{split}	
					\frac{\frac{1}{m_1}-\frac{1}{r_2}}{1-\frac{d}{2}+\frac{(p+m_1-1)d}{2m_1}}&+	 \frac{\frac{q}{m_2}-1+\frac{1}{r_2}}{1-\frac{d}{2}+\frac{(q
							+m_2-1)d}{2m_2}}\\
					&=\frac{\frac{q}{m_2}+\frac{1}{m_1}-1}{1+\frac{(p
							-1)d}{2m_1}}\\
					&=\frac{p+\frac{m_1}{m_2}-m_1}{p+\frac{2m_1}{d}-1}\cdot\frac{2}{d} <\frac{2}{d},
				\end{split}
			\end{equation*}
			which implies (\ref{choice of p and q 6})- (\ref{choice of p and q 7}).\\
			
			\noindent{\bf{Step 2.}}  \emph{Inequalities for both  $u$ and $w$.}
			For $p>1$ and $q>1$, we test $(\ref{TSTC approxiamation})_1$ by $u^{p-1}_\epsilon$ and integrate to find that
			\begin{equation*}
				\begin{split}
					\frac{1}{p}\frac{d}{dt}\int_{\mathbb{R}^d}u^p_{\epsilon} dx=&-(p-1)\int_{\mathbb{R}^d}u^{p-2}_{\epsilon}\nabla u_{\epsilon}\cdot\left(\nabla(u_{\epsilon}+\epsilon)^{m_1}-u_{\epsilon}\nabla v_{\epsilon}\right)dx\\
					\leq&-\frac{4m_1(p-1)}{(p+m_1-1)^2}\int_{\mathbb{R}^d}|\nabla u^{\frac{p+m_1-1}{2}}_{\epsilon}|^2dx-\frac{p-1}{p}\int_{\mathbb{R}^d}u^p_{\epsilon}\Delta v_{\epsilon}dx\\
					=&-\frac{4m_1(p-1)}{(p+m_1-1)^2}\int_{\mathbb{R}^d}|\nabla u^{\frac{p+m_1-1}{2}}_{\epsilon}|^2dx+\frac{p-1}{p}\int_{\mathbb{R}^d}u^{p}_{\epsilon}w_{\epsilon}dx
				\end{split}	
			\end{equation*}
			with $-\Delta v_{\epsilon}=w_{\epsilon}$, and similarly,
			\begin{equation*}
				\begin{split}
					 \frac{1}{q}\frac{d}{dt}\int_{\mathbb{R}^d}w^q_{\epsilon}dx\leq&-\frac{4m_2(q-1)}{(q+m_2-1)^2}\int_{\mathbb{R}^d}|\nabla w^{\frac{q+m_2-1}{2}}_{\epsilon}|^2dx+\frac{q-1}{q}\int_{\mathbb{R}^d}u_{\epsilon}w^{q}_{\epsilon}dx
				\end{split}
			\end{equation*}
			holds by multiplying $(\ref{TSTC approxiamation})_3$ by $w^{q-1}_\epsilon$ and $-\Delta z_{\epsilon}=u_{\epsilon}$. Then
			\be{\label{inequality for u^p and w^q}}
			\begin{split}
				\frac{1}{p}\frac{d}{dt}\int_{\mathbb{R}^d}u^p_{\epsilon} dx&+\frac{1}{q}\frac{d}{dt}\int_{\mathbb{R}^d}w^q_\epsilon dx+\frac{4m_1(p-1)}{(p+m_1-1)^2}\int_{\mathbb{R}^d}|\nabla u^{\frac{p+m_1-1}{2}}_\epsilon|^2dx\\
				&+\frac{4m_2(q-1)}{(q+m_2-1)^2}\int_{\mathbb{R}^d}|\nabla w^{\frac{q+m_2-1}{2}}_\epsilon|^2dx\\
				&\leq\frac{p-1}{p}\int_{\mathbb{R}^d}u^p_\epsilon w_\epsilon dx+\frac{q-1}{q}\int_{\mathbb{R}^d}u_\epsilon w^q_\epsilon dx,
			\end{split}
			\ee
			where
			\be{\label{right estimate for uw1}}
			\begin{split}
				\int_{\mathbb{R}^d}u^p_{\epsilon}w_{\epsilon}dx\leq \left(\int_{\mathbb{R}^d}u^{pr_1}_{\epsilon}dx\right)^{\frac{1}{r_1}}\left(\int_{\mathbb{R}^d}w^{r'_1}_{\epsilon}dx\right)^{\frac{1}{r'_1}}
			\end{split}
			\ee
			and
			\be{\label{right estimate for uw2}}
			\begin{split}
				\int_{\mathbb{R}^d}u_{\epsilon}w^q_{\epsilon}dx\leq \left(\int_{\mathbb{R}^d}u^{r_2}_{\epsilon}dx\right)^{\frac{1}{r_2}}\left(\int_{\mathbb{R}^d}w^{qr'_2}_{\epsilon}dx\right)^{\frac{1}{r'_2}}
			\end{split}
			\ee
			by H\"{o}lder's inequality with $r_1,r_2>1$, $r'_1=\frac{r_1}{r_1-1}$ and
			$r'_2=\frac{r_2}{r_2-1}$. We begin with estimating the right sides of ({\ref{right estimate for uw1}})-({\ref{right estimate for uw2}}) based on the choices of $p,q,r_1$ and $r_2$ in {\bf{Step 1}}. The assumption (\ref{choice of p and q 1}) ensures
			\be{\label{condition for u1}}
			\begin{split}
				pr_1>m_1,
			\end{split}
			\ee
			and
			\be{\label{condition for u2}}
			\begin{split}
				pr_1<\frac{(p+m_1-1)d}{d-2}
			\end{split}
			\ee
			by (\ref{choice of p and q 3}). Then by a variant of the Gagliardo-Nirenberg inequality (see {\cite[Lemma 6]{Sugiyama06-JDE}}),
			\be{\label{GN}}
			\|\varphi\|_{{k_2}}\leq C^{\frac{2}{r+m-1}}\|\varphi\|^{1-\sigma}_{{k_1}}\|\nabla \varphi^{\frac{r+m-1}{2}}\|^{\frac{2\sigma}{r+m-1}}_{2}
			\ee
			with $m\geq 1$, $k_1\in[1,r+m-1]$ and $1\leq k_1\leq k_2\leq \frac{(r+m-1)d}{d-2}$ with $d\geq 3$, $\sigma=\frac{r+m-1}{2}\left(\frac{1}{k_1}-\frac{1}{k_2}\right)\left(\frac{1}{d}-\frac{1}{2}+\frac{r+m-1}{2k_1}\right)^{-1}$,
			we pick $r=p, m=m_1$, $k_1=m_1$, $k_2=pr_1$ in (\ref{GN}) and use (\ref{condition for u1})-(\ref{condition for u2}) to find
			\begin{equation*}
				\begin{split}
					\left(\int_{\mathbb{R}^d}u^{pr_1}_{\epsilon}dx\right)^{\frac{1}{r_1}}=\|u_{\epsilon}\|^{p} _{pr_1}\leq C\|u_{\epsilon}\|^{p(1-\sigma)}_{m_1}\|\nabla u^{\frac{p+m_1-1}{2}}_{\epsilon}\|^{p\frac{2\sigma}{p+m_1-1}}_{2}
				\end{split}
			\end{equation*}
			with
			\begin{equation*}
				\begin{split}
					 \sigma=\frac{p+m_1-1}{2}\frac{\frac{1}{m_1}-\frac{1}{pr_1}}{\frac{1}{d}-\frac{1}{2}+\frac{p+m_1-1}{2m_1}}\in(0,1),
				\end{split}
			\end{equation*}
			where invoking (\ref{Asssumption on prior estimate for uwepsilon}) we further obtain
			\begin{equation*}
				\begin{split}
					\left(\int_{\mathbb{R}^d}u^{pr_1}_{\epsilon}dx\right)^{\frac{1}{r_1}}\leq C\|\nabla u^{\frac{p+m_1-1}{2}}_{\epsilon}\|^{\frac{\frac{p}{m_1}-\frac{1}{r_1}}{\frac{1}{d}-\frac{1}{2}+\frac{p+m_1-1}{2m_1}}}_{2}.
				\end{split}
			\end{equation*}
			Likewise, (\ref{choice of p and q 2})-(\ref{choice of p and q 3}) warrants that
			\begin{equation*}{\label{condition for u3}}
				\begin{split}
					m_2<r'_1<\frac{(q+m_2-1)d}{d-2},
				\end{split}
			\end{equation*}
			which allows one to make use of the Gagliardo-Nirenberg inequality and the upper bound for $\|w\|_{m_2}$ in (\ref{Asssumption on prior estimate for uwepsilon}) to estimate
			\begin{equation*}
				\begin{split}
					\left(\int_{\mathbb{R}^d}w^{r'_1}_{\epsilon}dx\right)^{\frac{1}{r'_1}}=\|w_{\epsilon}\| _{r'_1}\leq C\|\nabla w^{\frac{q+m_2-1}{2}}_{\epsilon}\|^{\frac{\frac{1}{m_2}-\frac{1}{r'_1}}{\frac{1}{d}-\frac{1}{2}+\frac{q+m_2-1}{2m_2}}}_{2}.
				\end{split}
			\end{equation*}
			Then
			\be{\label{inequality for u^p and w^p3}}
			\begin{split}
				 \left(\int_{\mathbb{R}^d}u^{pr_1}_{\epsilon}dx\right)^{\frac{1}{r_1}}&\left(\int_{\mathbb{R}^d}w^{r'_1}_{\epsilon}dx\right)^{\frac{1}{r'_1}}\\
				&\leq C\|\nabla u^{\frac{p+m_1-1}{2}}_{\epsilon}\|^{\frac{\frac{p}{m_1}-\frac{1}{r_1}}{\frac{1}{d}-\frac{1}{2}+\frac{p+m_1-1}{2m_1}}}_{2}\cdot\|\nabla w^{\frac{q+m_2-1}{2}}_{\epsilon}\|^{\frac{\frac{1}{m_2}-\frac{1}{r'_1}}{\frac{1}{d}-\frac{1}{2}+\frac{q+m_2-1}{2m_2}}}_{2}\\
				&= C\|\nabla u^{\frac{p+m_1-1}{2}}_{\epsilon}\|^{\frac{\frac{p}{m_1}-\frac{1}{r_1}}{\frac{1}{d}-\frac{1}{2}+\frac{p+m_1-1}{2m_1}}}_{2}\cdot\|\nabla w^{\frac{q+m_2-1}{2}}_{\epsilon}\|^{\frac{\frac{1}{m_2}-1+\frac{1}{r_1}}{\frac{1}{d}-\frac{1}{2}+\frac{q+m_2-1}{2m_2}}}_{2}.
			\end{split}
			\ee
			To estimate the right side of ({\ref{right estimate for uw2}}), we use (\ref{choice of p and q 5}) and (\ref{choice of p and q 4}) to obtain
			\begin{equation*}{\label{condition for u4}}
				\begin{split}
					m_1<r_2<\frac{(p+m_1-1)d}{d-2}.
				\end{split}
			\end{equation*}
			Then  the Gagliardo-Nirenberg inequality implies
			\begin{equation*}
				\begin{split}
					\left(\int_{\mathbb{R}^d}u^{r_2}_{\epsilon}dx\right)^{\frac{1}{r_2}}
					\leq& C\|\nabla u^{\frac{p+m_1-1}{2}}_{\epsilon}\|^{\frac{\frac{1}{m_1}-\frac{1}{r_2}}{\frac{1}{d}-\frac{1}{2}+\frac{p+m_1-1}{2m_1}}}_{2}
				\end{split}
			\end{equation*}
			by (\ref{Asssumption on prior estimate for uwepsilon}).
			We also obtain
			\begin{equation*}
				\begin{split}
					m_2<qr'_2<\frac{(q+m_2-1)d}{d-2}
				\end{split}
			\end{equation*}
			by (\ref{choice of p and q 5}) and (\ref{defintion of q}), and  choose $r=q, m=m_2$, $k_1=m_2$, $k_2=qr'_2$  in
			(\ref{GN})  to see that
			\begin{equation*}
				\begin{split}
					\left(\int_{\mathbb{R}^d}w^{qr'_2}_{\epsilon}dx\right)^{\frac{1}{r'_2}}=&\|w_{\epsilon}\|^{q} _{qr'_2}\leq C\|w_{\epsilon}\|^{q(1-\sigma)}_{m_2}\|\nabla w^{\frac{q+m_2-1}{2}}_{\epsilon}\|^{q\frac{2\sigma}{q+m_2-1}}_{2}\\
					\leq& C\|\nabla w^{\frac{q+m_2-1}{2}}_{\epsilon}\|^{\frac{\frac{q}{m_2}-\frac{1}{r'_2}}{\frac{1}{d}-\frac{1}{2}+\frac{q+m_2-1}{2m_2}}}_{2}\\
					=&C\|\nabla w^{\frac{q+m_2-1}{2}}_{\epsilon}\|^{\frac{\frac{q}{m_2}-1+\frac{1}{r_2}}{\frac{1}{d}-\frac{1}{2}+\frac{q+m_2-1}{2m_2}}}_{2}
				\end{split}
			\end{equation*}
			with
			\begin{equation*}
				\begin{split}
					 \sigma=\frac{q+m_2-1}{2}\frac{\frac{1}{m_2}-\frac{1}{qr'_2}}{\frac{1}{d}-\frac{1}{2}+\frac{q+m_2-1}{2m_2}}.
				\end{split}
			\end{equation*}
			Then
			\begin{equation*}
				\begin{split}
					 \left(\int_{\mathbb{R}^d}u^{r_2}_{\epsilon}dx\right)^{\frac{1}{r_2}}&\left(\int_{\mathbb{R}^d}w^{qr'_2}_{\epsilon}dx\right)^{\frac{1}{r'_2}}\\
					&\leq C\|\nabla u^{\frac{p+m_1-1}{2}}_{\epsilon}\|^{\frac{\frac{1}{m_1}-\frac{1}{r_2}}{\frac{1}{d}-\frac{1}{2}+\frac{p+m_1-1}{2m_1}}}_{2}\|\nabla w^{\frac{q+m_2-1}{2}}_{\epsilon}\|^{\frac{\frac{q}{m_2}-1+\frac{1}{r_2}}{\frac{1}{d}-\frac{1}{2}+\frac{q+m_2-1}{2m_2}}}_{2},
				\end{split}
			\end{equation*}
			which combines with  (\ref{inequality for u^p and w^q}) and (\ref{inequality for u^p and w^p3}) ensures that
			\be{\label{inequality for u^p and w^q2}}
			\begin{split}
				 \frac{1}{p}\frac{d}{dt}\int_{\mathbb{R}^d}u^p_{\epsilon}dx&+\frac{1}{q}\frac{d}{dt}\int_{\mathbb{R}^d}w^q_{\epsilon}dx+\frac{4m_1(p-1)}{(p+m_1-1)^2}\int_{\mathbb{R}^d}|\nabla u^{\frac{p+m_1-1}{2}}_{\epsilon}|^2dx\\
				&+\frac{4m_2(q-1)}{(q+m_2-1)^2}\int_{\mathbb{R}^d}|\nabla w^{\frac{q+m_2-1}{2}}_{\epsilon}|^2dx\\
				\leq& \frac{p-1}{p}\left(\int_{\mathbb{R}^d}u^{pr_1}_{\epsilon}dx\right)^{\frac{1}{r_1}}\left(\int_{\mathbb{R}^d}w^{r'_1}_{\epsilon}
				dx\right)^{\frac{1}{r'_1}}\\
				 &+\frac{q-1}{q}\left(\int_{\mathbb{R}^d}u^{r_2}_{\epsilon}dx\right)^{\frac{1}{r_2}}\left(\int_{\mathbb{R}^d}w^{qr'_2}_{\epsilon}dx\right)^{\frac{1}{r'_2}}\\
				\leq&C\|\nabla u^{\frac{p+m_1-1}{2}}_{\epsilon}\|^{\frac{\frac{p}{m_1}-\frac{1}{r_1}}{\frac{1}{d}-\frac{1}{2}+\frac{p+m_1-1}{2m_1}}}_{2}\cdot\|\nabla w^{\frac{q+m_2-1}{2}}_{\epsilon}\|^{\frac{\frac{1}{m_2}-1+\frac{1}{r_1}}{\frac{1}{d}-\frac{1}{2}+\frac{q+m_2-1}{2m_2}}}_{2}\\
				&+C\|\nabla u^{\frac{p+m_1-1}{2}}_{\epsilon}\|^{\frac{\frac{1}{m_1}-\frac{1}{r_2}}{\frac{1}{d}-\frac{1}{2}+\frac{p+m_1-1}{2m_1}}}_{2}\cdot\|\nabla w^{\frac{q+m_2-1}{2}}_{\epsilon}\|^{\frac{\frac{q}{m_2}-1+\frac{1}{r_2}}{\frac{1}{d}-\frac{1}{2}+\frac{q+m_2-1}{2m_2}}}_{2}.
			\end{split}
			\ee
			
			\noindent{\bf{Step 3.}} \emph{Boundedness for $u_{\epsilon}$ and $w_{\epsilon}$ in $L^p$- and $L^q$- spaces}. Let $\gamma_1>0,\gamma_2>0$ be such that $\gamma_1+\gamma_2<2$. For  $\epsilon>0$, a direct application of Young's inequality implies that
			\be\label{inequality for gamma1 and gamma2}
			\begin{split}
				\alpha^{\gamma_1}\beta^{\gamma_2}\leq \epsilon (\alpha^{2}+\beta^{2})+C.
			\end{split}
			\ee
			From {\bf{Step 1}}, there exist some $p>\bar{p}$ and $q>\bar{q}$ with some $\bar{p}>1$ and $\bar{q}>1$ such that
			\begin{equation*}
				\begin{split}
					 \frac{\frac{p}{m_1}-\frac{1}{r_1}}{\frac{1}{d}-\frac{1}{2}+\frac{p+m_1-1}{2m_1}}+\frac{\frac{1}{m_2}-1+\frac{1}{r_1}}{\frac{1}{d}-\frac{1}{2}+\frac{q+m_2-1}{2m_2}}<2
				\end{split}
			\end{equation*}
			and
			\begin{equation*}
				\begin{split}
					 \frac{\frac{1}{m_1}-\frac{1}{r_2}}{\frac{1}{d}-\frac{1}{2}+\frac{p+m_1-1}{2m_1}}+\frac{\frac{q}{m_2}-1+\frac{1}{r_2}}{\frac{1}{d}-\frac{1}{2}+\frac{q+m_2-1}{2m_2}}<2,
				\end{split}
			\end{equation*}
			where
			\begin{equation}{\label{inequality for u^p and w^q3}}
				\begin{split}
					 \frac{1}{p}\frac{d}{dt}\int_{\mathbb{R}^d}u^p_{\epsilon}dx&+\frac{1}{q}\frac{d}{dt}\int_{\mathbb{R}^d}w^q_{\epsilon}dx+\frac{2m_1(p-1)}{(p+m_1-1)^2}\int_{\mathbb{R}^d}|\nabla u^{\frac{p+m_1-1}{2}}_{\epsilon}|^2dx\\
					&+\frac{2m_2(q-1)}{(q+m_2-1)^2}\int_{\mathbb{R}^d}|\nabla w^{\frac{q+m_2-1}{2}}_{\epsilon}|^2dx\leq C
				\end{split}
			\end{equation}
			by (\ref{inequality for u^p and w^q2})-(\ref{inequality for gamma1 and gamma2}). One may invoke
			the Gagliardo-Nirenberg inequality with $\|u\|_{1}=M_1$ and $\|w\|_{1}=M_2$ and Young's inequality to obtain
			\begin{equation*}
				\begin{split}
					\frac{1}{p}\int_{\mathbb{R}^d}u^{p}_{\epsilon}dx=\frac{1}{p}\|u_{\epsilon}\|^{p} _{p}&\leq C\|\nabla u^{\frac{p+m_1-1}{2}}_{\epsilon}\|^{\frac{p-1}{\frac{1}{d}-\frac{1}{2}+\frac{p+m_1-1}{2}}}_{2}\\
					&\leq \frac{2m_1(p-1)}{(p+m_1-1)^2}\int_{\mathbb{R}^d}|\nabla u^{\frac{p+m_1-1}{2}}_{\epsilon}|^2dx+C
				\end{split}
			\end{equation*}
			and
			\begin{equation*}
				\begin{split}
					 \frac{1}{q}\int_{\mathbb{R}^d}w^{q}_{\epsilon}dx\leq\frac{2m_2(q-1)}{(q+m_2-1)^2}\int_{\mathbb{R}^d}|\nabla w^{\frac{q+m_2-1}{2}}_{\epsilon}|^2dx+C
				\end{split}
			\end{equation*}
			by the fact that
			\begin{equation*}
				\begin{split}
					 \frac{p-1}{\frac{1}{d}-\frac{1}{2}+\frac{p+m_1-1}{2}}<2\,\,\,\text{and}\,\,\,\,\frac{q-1}{\frac{1}{d}-\frac{1}{2}+\frac{q+m_2-1}{2}}<2.
				\end{split}
			\end{equation*}
			Writing $y(t)=\frac{1}{p}\int_{\mathbb{R}^d}u^{p}_{\epsilon}dx+\frac{1}{q}\int_{\mathbb{R}^d}w^{q}_{\epsilon}dx$, we obtain from (\ref{inequality for u^p and w^q3}) that
			\begin{equation*}
				\begin{split}
					y'(t)+y(t)\leq C\,\,\,\,\text{for}\,\,\,\,t\in(0,T).
				\end{split}
			\end{equation*}
			Then
			\begin{equation*}{\label{inequality for u^p and w^q4}}
				\begin{split}
					\|u_{\epsilon}(t)\|_p\leq C\,\,\,\,\text{and}\,\,\,\,\|w_{\epsilon}(t)\|_q\leq C\,\,\,\,\text{for}\,\,\,\,t\in(0,T),
				\end{split}
			\end{equation*}
			which implies that (\ref{prior estimate for uwepsilon2}) holds.\\
			\noindent{\bf{Step 4.}}  \emph{Improve the regularities of $v$ and $z$.} As
			\begin{equation*}
				\begin{split}
					v_{\epsilon}=\mathcal{K}\ast w_{\epsilon}=c_d\int_{\mathbb{R}^d}\frac{w_{\epsilon}(y)}{|x-y|^{d-2}}dy,\,\,z_{\epsilon}=c_d\int_{\mathbb{R}^d}\frac{u_{\epsilon}(y)}{|x-y|^{d-2}}dy,
				\end{split}
			\end{equation*}
			an application of the HLS inequality ensures  that
			\be\label{gradient for vepislon}
			\begin{split}
				\||\nabla v_{\epsilon}|\|_r\leq&c_d(d-2)\left\|I_1(w_{\epsilon})\right\|_r
				\leq C\left\|w_{\epsilon}\right\|_{dr/(d+r)},\\
				\||\nabla z_{\epsilon}|\|_{r}\leq& C\left\|u_{\epsilon}\right\|_{dr/(d+r)}.
			\end{split}
			\ee
			Furthermore, observing that the Calderon-Zygmund inequality yields the existence of a constant $C=C(r)>0$ such that
			\begin{equation*}
				\begin{split}
					\|\partial_{x_i}\partial_{x_j} v_{\epsilon}\|_{r}\leq& C\left\|w_{\epsilon}\right\|_{r},\\
					\|\partial_{x_i}\partial_{x_j} z_{\epsilon}\|_{r}\leq& C\left\|u_{\epsilon}\right\|_{r},\,\,\,(1\leq i,j\leq d),
				\end{split}
			\end{equation*}
			we combine (\ref{prior estimate for uwepsilon2}), (\ref{gradient for vepislon}) with the Morrey's inequality to see that
			\begin{equation*}
				\begin{split}
					\|(v_{\epsilon}(t),z_{\epsilon}(t))\|_{r}+\|(\nabla v_{\epsilon}(t),\nabla z_{\epsilon}(t))\|_{r}\leq  C\,\,\,\text{for}\,\,\,r\in[1,\infty]\,\,\,\text{and}\,\,\,t\in(0,T).	
				\end{split}
			\end{equation*}
			Thus we finish our proof.
		\end{proof}
		
		Upon the boundedness arguments in Lemma {\ref{the condition for global existence}}, we obtain a global weak solution by letting a subsequence of $\epsilon$ approaches to 0.

		\begin{lemma} \label{global weak solution exist}
			Under the same assumption in Lemma \ref{the condition for global existence}, there exists $C>0$ independent of $\epsilon$ such that the strong solution $(u_{\epsilon},w_{\epsilon})$ of (\ref{TSTC approxiamation}) satisfies
			\be\label{Linfty estimate for uwepsilon}
			\|(u_{\epsilon}(t),w_{\epsilon}(t))\|_{\infty}\leq  C\,\,\,\,\text{for all}\,\,\,\,t\in(0,T).
			\ee
			Moreover, there exists a global weak solution  $(u,w)$ of (\ref{TSTC}) which also satisfies  a uniform estimate.
		\end{lemma}
		\begin{proof}
			Relying on Lemma \ref{the condition for global existence}, we apply the Moser's iteration technique to obtain a priori estimate of solution in $L^{\infty}$. Then  this local solution can be extended globally in time from the extensibility criterion in Lemma {\ref{Local existence for AS}}, which indeed establishes (\ref{Linfty estimate for uwepsilon}), see {\cite[Proposition 10]{Sugiyama06-JDE}}. Moreover, from (\ref{Linfty estimate for uwepsilon}) there exists $(u,v,w,z)$ with the regularities given in Definition \ref{weak solution} such that, up to a subsequence, $\epsilon_{n}\rightarrow 0$,
			\begin{align*}
				u_{\epsilon_{n}}\rightarrow& u\,\,\text{strongly in}\,\,C([0,T);L^{p}_{loc}(\mathbb{R}^d))\,\,\text{and a.e. in}\,\,\mathbb{R}^d\times (0,T),\\
				\nabla u^{m_1}_{\epsilon_{n}}\rightharpoonup&\nabla u^{m_1}\,\,\text{weakly-}\ast\,\,{in}\,\,L^\infty((0,T); L^2(\mathbb{R}^d)),\\
				w_{\epsilon_{n}}\rightarrow& w\,\,\text{strongly in}\,\,C([0,T);L^{p}_{loc}(\mathbb{R}^d))\,\,\text{and a.e. in}\,\,\mathbb{R}^d\times (0,T),\\
				\nabla w^{m_2}_{\epsilon_{n}}\rightharpoonup&\nabla w^{m_2}\,\,\text{weakly-}\ast\,\,{in}\,\,L^\infty((0,T); L^2(\mathbb{R}^d)),\\
				v_{\epsilon_{n}}(t)\rightarrow& v(t)\,\,\text{strongly in}\,\,L^{r}_{loc}(\mathbb{R}^d)\,\,\text{and a.e. in}\,\,(0,T),\\
				\nabla v_{\epsilon_{n}}(t)\rightarrow&\nabla v(t)\,\,\text{strongly in}\,\, L^r_{loc}(\mathbb{R}^d)\,\,\text{and a.e. in}\,\,(0,T),\\
				\Delta 	v_{\epsilon_{n}}(t)\rightharpoonup&\Delta v(t)\,\,\text{weakly in}\,\,L^{r}_{loc}(\mathbb{R}^d)\,\,\text{and a.e. in}\,\,(0,T),\\
				z_{\epsilon_{n}}(t)\rightarrow& z(t)\,\,\text{strongly in}\,\,L^{r}_{loc}(\mathbb{R}^d)\,\,\text{and a.e. in}\,\,(0,T),\\
				\nabla z_{\epsilon_{n}}(t)\rightarrow&\nabla z(t)\,\,\text{strongly in}\,\, L^r_{loc}(\mathbb{R}^d)\,\,\text{and a.e. in}\,\,(0,T),\\
				\Delta 	z_{\epsilon_{n}}(t)\rightharpoonup&\Delta z(t)\,\,\text{weakly in}\,\,L^{r}_{loc}(\mathbb{R}^d)\,\,\text{and a.e. in}\,\,(0,T),
			\end{align*}
			where $p\in(1,\infty)$, $r\in(1,\infty]$  and $T\in (0,\infty)$. Since the above convergence can be calculated in {\cite[Section 4]{Sugiyama07-ADE}}, we omit the main proof here. Therefore, we have a global weak solution $(u,v,w,z)$ over $\mathbb{R}^d\times(0,T)$ with $T>0$.

		\end{proof}
		
		The weak solution obtained in Lemma \ref{global weak solution exist} is also a free energy solution given in Definition {\ref{free energy solution}}. The proof comes from \cite{Sugiyama06-DIE}.
		\begin{lemma}  \label{proof of FES} Consider a global weak solution in Lemma \ref{global weak solution exist}, then it is also a global free energy solution $(u,w)$ of (\ref{TSTC}) given in Definition {\ref{free energy solution}}.
		\end{lemma}
		\begin{proof}
			Define a weight function
			\begin{equation*}\label{weight function}
				\begin{split}
					\psi (|x|)=\left\{
					\begin{array}{llll}
						1,\,\,&\text{for}\,\,0\leq |x|\leq 1,\\
					1-2(|x|-1)^2,\,\,&\text{for}\,\,1< |x|\leq \frac{3}{2},\\
					2(2-|x|)^2,\,\,&\text{for}\,\,\frac{3}{2}< |x|<2,\\
					0,\,\,&\text{for}\,\,|x|\geq2,
					\end{array}
					\right.
				\end{split}
			\end{equation*}
			and define $\psi_l$ by $\psi_{l}(x):=\psi\left(\frac{|x|}{l}\right)$ for any $x\in\mathbb{R}^d$ and $l=1,2,3,\cdots. $ 	Evidently,
			\begin{equation*}\label{property for weight function}
				\begin{split}
					|\nabla \psi_{l}(x)|\leq \frac{C}{l}(\psi_l(x))^{\frac{1}{2}}\,\,\,\text{and}\,\,
					|\Delta \psi_{l}(x)| \leq \frac{C}{l^2}
				\end{split}
			\end{equation*}
			is valid with some $C>0$. Denote
			\begin{equation*}
				\begin{split}
					\mathcal{F}[u_{\epsilon}(t),w_{\epsilon}(t)]:=&\frac{1}{m_1-1}\int_ {\mathbb{R}^d}(u_{\epsilon}+\epsilon)^{m_1}\psi_{l}(x)dx+\frac{1}{m_2-1}\int_ {\mathbb{R}^d}(w_{\epsilon}+\epsilon)^{m_2}\psi_{l}(x)dx\\
					&-\int_ {\mathbb{R}^d}u_{\epsilon}v_{\epsilon}dx\\
					=&\frac{1}{m_1-1}\int_ {\mathbb{R}^d}(u_{\epsilon}+\epsilon)^{m_1}\psi_{l}(x)dx-\int_ {\mathbb{R}^d}u_{\epsilon}v_{\epsilon}dx\\	
					&+\frac{1}{m_2-1}\int_ {\mathbb{R}^d}(w_{\epsilon}+\epsilon)^{m_2}\psi_{l}(x)dx-\int_ {\mathbb{R}^d}w_{\epsilon}z_{\epsilon}dx\\
					&+\int_ {\mathbb{R}^d}\nabla v_{\epsilon}\cdot \nabla z_{\epsilon}dx.
				\end{split}
			\end{equation*}
			Since
			\begin{equation*}
				\begin{split}
					 \frac{1}{m_1-1}&\frac{d}{dt}(u_{\epsilon}+\epsilon)^{m_1}\psi_{l}-\frac{d}{dt}(u_{\epsilon}v_{\epsilon})+u_{\epsilon}v_{\epsilon t}\\
					&=\nabla\cdot \left(\nabla(u_{\epsilon}+\epsilon)^{m_1}-u_{\epsilon}\nabla v_{\epsilon}\right)\cdot\left(\frac{m_1(u_{\epsilon}+\epsilon)^{m_1-1}}{m_1-1}\psi_{l}-v_{\epsilon}\right),\\
					 \frac{1}{m_2-1}&\frac{d}{dt}(w_{\epsilon}+\epsilon)^{m_2}\psi_{l}-\frac{d}{dt}(w_{\epsilon}z_{\epsilon})+w_{\epsilon}z_{\epsilon t}\\
					&=\nabla\cdot \left(\nabla(w_{\epsilon}+\epsilon)^{m_2}-w_{\epsilon}\nabla z_{\epsilon}\right)\cdot\left(\frac{m_2(w_{\epsilon}+\epsilon)^{m_2-1}}{m_2-1}\psi_{l}-z_{\epsilon}\right)
				\end{split}
			\end{equation*}
			by testing $(\ref{TSTC approxiamation})_1$ by $\frac{m_1(u_{\epsilon}+\epsilon)^{m_1-1}}{m_1-1}\psi_{l}-v_{\epsilon}$ and $(\ref{TSTC approxiamation})_3$ by $\frac{m_2(w_{\epsilon}+\epsilon)^{m_2-1}}{m_2-1}\psi_{l}-z_{\epsilon}$, then the derivative of $	 \mathcal{F}[u_{\epsilon}(t),w_{\epsilon}(t)]$ with respect to time is
			\begin{equation*}\label{}
				\begin{split}
					\frac{d}{dt}\mathcal{F}[u_{\epsilon}(t),w_{\epsilon}(t)]=&\frac{1}{m_1-1}\frac{d}{dt}\int_ {\mathbb{R}^d}(u_{\epsilon}+\epsilon)^{m_1}\psi_{l}(x)dx-\frac{d}{dt}\int_ {\mathbb{R}^d}u_{\epsilon}v_{\epsilon}dx\\
					&+\int_ {\mathbb{R}^d}u_{\epsilon}v_{\epsilon t}dx+\frac{1}{m_2-1}\frac{d}{dt}\int_ {\mathbb{R}^d}(w_{\epsilon}+\epsilon)^{m_2}\psi_{l}(x)dx
					\\
					&-\frac{d}{dt}\int_ {\mathbb{R}^d}w_{\epsilon}z_{\epsilon}dx+\int_ {\mathbb{R}^d}w_{\epsilon}z_{\epsilon t}dx\\
					=&-\int_ {\mathbb{R}^d} \left(\nabla(u_{\epsilon}+\epsilon)^{m_1}-u_{\epsilon}\nabla v_{\epsilon}\right)\cdot\nabla\left(\frac{m_1(u_{\epsilon}+\epsilon)^{m_1-1}}{m_1-1}\psi_{l}-v_{\epsilon}\right)dx\\
					&-\int_ {\mathbb{R}^d}\left(\nabla(w_{\epsilon}+\epsilon)^{m_2}-w_{\epsilon}\nabla z_{\epsilon}\right)\cdot\nabla \left(\frac{m_2(w_{\epsilon}+\epsilon)^{m_2-1}}{m_2-1}\psi_{l}-z_{\epsilon}\right)dx,
				\end{split}
			\end{equation*}
			which can be written as
			\begin{equation}\label{derivative of F}
				\begin{split}	
					 \frac{d}{dt}\mathcal{F}[u_{\epsilon}(t),w_{\epsilon}(t)]=&-\int_{\mathbb{R}^d}\left[(u_\epsilon+\epsilon)\nabla\left(\frac{m_1}{m_1-1}(u_\epsilon+\epsilon)^{m_1-1}-v_\epsilon\right)+\epsilon\nabla v_\epsilon\right]\\
					&\cdot\bigg[\nabla\left(\frac{m_1}{m_1-1}(u_\epsilon+\epsilon)^{m_1-1}-v_\epsilon\right)\psi_l\\
					&+\left(\frac{m_1}{m_1-1}(u_\epsilon+\epsilon)^{m_1-1}-v_\epsilon\right)\nabla \psi_l+\nabla v_\epsilon(\psi_l-1)+v_\epsilon\nabla \psi_l\bigg]dx\\
					 &-\int_{\mathbb{R}^d}\left[(w_\epsilon+\epsilon)\nabla\left(\frac{m_2}{m_2-1}(w_\epsilon+\epsilon)^{m_2-1}-z_\epsilon\right)+\epsilon\nabla z_\epsilon\right]\\
					&\cdot\bigg[\nabla\left(\frac{m_2}{m_2-1}(w_\epsilon+\epsilon)^{m_2-1}-z_\epsilon\right)\psi_l\\
					&+\left(\frac{m_2}{m_2-1}(w_\epsilon+\epsilon)^{m_2-1}-z_\epsilon\right)\nabla \psi_l+\nabla z_\epsilon(\psi_l-1)+z_\epsilon\nabla \psi_l\bigg]dx\\
					=& -\int_{\mathbb{R}^d}I_1\times J_1dx-\int_{\mathbb{R}^d}I_2\times J_2dx.
				\end{split}
			\end{equation}
			With $U_\epsilon:=\frac{m_1}{m_1-1}(u_\epsilon+\epsilon)^{m_1-1}-v_\epsilon$,
			we expand the term $-\int_{\mathbb{R}^d}I_1\times J_1dx$ to find that
			\begin{equation*}\label{3.22}
				\begin{split}
					-\int_{\mathbb{R}^d}I_1\times J_1dx=&-\int_{\mathbb{R}^d}(u_\epsilon+\epsilon)\psi_l|\nabla U_\epsilon|^2dx-\int_{\mathbb{R}^d}(u_\epsilon+\epsilon)(U_\epsilon+v_\epsilon)\nabla U_\epsilon\cdot\nabla \psi_ldx\\
					&-\int_{\mathbb{R}^d}(u_\epsilon+\epsilon)(\psi_l-1)\nabla U_\epsilon\cdot\nabla v_\epsilon dx-\epsilon\int_{\mathbb{R}^d}\nabla(\psi_l U_\epsilon)\cdot \nabla v_{\epsilon}dx\\
					&-\epsilon\int_{\mathbb{R}^d}\nabla (v_{\epsilon}(\psi_l-1))\cdot \nabla v_{\epsilon}dx,
				\end{split}
			\end{equation*}
			where  by defining  $\Omega_{l}=\{x\in \mathbb{R}^d:l<|x|<2l\}$, upon  using Young's inequality, H\"{o}lder's inequality and $(a+b)^m\leq 2^m(a^m+b^m)$ with $a,b>0$ and $m>1$, with any $\eta\in(0,1)$ we deduce  from $|\nabla \psi_l|\leq \frac{C}{l}\left(\psi_l\right)^{\frac{1}{2}}$  and supp$|\nabla \psi_l|=\overline{\Omega}_l$ that
			\begin{equation*}\label{3.23}
				\begin{split}
					-\int_{\mathbb{R}^d}(u_\epsilon+\epsilon)(U_\epsilon+v_\epsilon)\nabla U_\epsilon\cdot\nabla \psi_ldx\leq& \eta\int_{\mathbb{R}^d}(u_\epsilon+\epsilon)\psi_l|\nabla U_\epsilon|^2dx\\
					&+\frac{C}{\eta l^2}\left(\|u_\epsilon\|^{2m_1-1}_{2m_1-1}+\epsilon^{2m_1-1}|\Omega_l|\right),\\
					-\int_{\mathbb{R}^d}(u_\epsilon+\epsilon)(\psi_l-1)\nabla U_\epsilon\cdot\nabla v_\epsilon dx=& \int_{\mathbb{R}^d}(1-\psi_l)\nabla (u_\epsilon+\epsilon)^{m_1}\cdot \nabla v_{\epsilon}dx\\
					&+\int_{\mathbb{R}^d}(u_\epsilon+\epsilon)(\psi_l-1) |\nabla v_{\epsilon}|^2dx\\
					\leq&\int_{\mathbb{R}^d}(u_\epsilon+\epsilon)^{m_1}\nabla \psi_l \cdot \nabla v_{\epsilon}dx\\
					&-\int_{\mathbb{R}^d}(1-\psi_l) (u_\epsilon+\epsilon)^{m_1}\Delta v_{\epsilon}dx\\ \leq&\int_{\mathbb{R}^d}(u_\epsilon+\epsilon)^{m_1}w_{\epsilon}(1-\psi_l)dx\\
					&+\frac{C}{l}\int_{\mathbb{R}^d}\left(u^{m_1}+\epsilon^{m_1}\right)|\nabla v_{\epsilon}|dx,\\
					-\epsilon\int_{\mathbb{R}^d}\nabla(\psi_l U_\epsilon)\cdot \nabla v_{\epsilon}dx=-\epsilon\int_{\mathbb{R}^d}\psi_l U_\epsilon w_{\epsilon}dx\leq&\epsilon \|w_{\epsilon}\|_{L^1}\|U_\epsilon\|_{L^\infty},\\
					-\epsilon\int_{\mathbb{R}^d}\nabla (v_{\epsilon}(\psi_l-1))\cdot \nabla v_{\epsilon}dx\leq& \epsilon\int_{\mathbb{R}^d}w_{\epsilon}v_{\epsilon}(1-\psi_l)dx.
				\end{split}
			\end{equation*}
			The regularities of $(u_{\epsilon},v_{\epsilon},w_{\epsilon})$ from Lemmas \ref{the condition for global existence}-\ref{global weak solution exist} assert that
			\begin{equation*}
				\begin{split}
					-\int_{\mathbb{R}^d}I_1\times J_1dx\leq&-(1-\eta)\int_{\mathbb{R}^d}(u_\epsilon+\epsilon)\psi_l|\nabla U_\epsilon|^2dx\\
					&+\frac{C}{\eta l^2}\left(\|u_\epsilon(t)\|^{2m_1-1}_{2m_1-1}+\epsilon^{2m_1-1}|\Omega_l|\right)\\
					 &+\int_{\mathbb{R}^d}(u_\epsilon+\epsilon)^{m_1}w_\epsilon(1-\psi_l)dx+\frac{C}{l}\int_{\mathbb{R}^d}\left(u^{m_1}_\epsilon+\epsilon^{m_1}\right)|\nabla v_{\epsilon}|dx\\
					&+\epsilon \|w_\epsilon\|_{L^1}\|U_\epsilon\|_{L^\infty}+\epsilon\int_{\mathbb{R}^d}w_\epsilon v_\epsilon(1-\psi_l)dx\\
					\leq&-(1-\eta)\int_{\mathbb{R}^d}(u_\epsilon+\epsilon)\psi_l|\nabla U_\epsilon|^2dx+\frac{C}{\eta l^2}\left(1+\epsilon^{2m_1-1}|\Omega_l|\right)\\
					&+C\int_{\mathbb{R}^d}w_\epsilon(1-\psi_l)dx+\frac{C}{l}+ \epsilon C.
				\end{split}
			\end{equation*}
			Doing a similar argument for $-\int_{\mathbb{R}^d}I_2\times J_2dx$, and integrating (\ref{derivative of F}) over time shows that
			\begin{equation*}\label{3.26}
				\begin{split}
					\mathcal{F}[u_{\epsilon}(t),w_{\epsilon}(t)]\leq&\mathcal{F}[u_{0\epsilon},w_{0\epsilon}]\\
					&-(1-\eta)\int^t_0\int_{\mathbb{R}^d}(u_\epsilon+\epsilon)\psi_l\left| \frac{m_1}{m_1-1}\nabla(u_\epsilon+\epsilon)^{m_1-1}-\nabla v_\epsilon\right|^2\\
					&-(1-\eta)\int^t_0 \int_{\mathbb{R}^d}(w_\epsilon+\epsilon)\psi_l\left| \frac{m_2}{m_2-1}\nabla(w_\epsilon+\epsilon)^{m_2-1}-\nabla z_\epsilon\right|^2\\
					&+\frac{CT}{\eta l^2}\left(1+\epsilon^{2m_1-1}|\Omega_l|+\epsilon^{2m_2-1}|\Omega_l|\right)\\
					&+C\int^t_0\int_{\mathbb{R}^d}(u_\epsilon+w_\epsilon)(1-\psi_l)+\frac{CT}{l}+\epsilon CT\,\,\,\text{for}\,\,t\in(0,T),
				\end{split}
			\end{equation*}
			where as $\epsilon$ tends to 0,
			\begin{equation*}\label{3.27}
				\begin{split}
					\mathcal{F}[u(t),w(t)]\leq&\mathcal{F}[u_{0},w_{0}]
					-(1-\eta)\int^t_0\int_{\mathbb{R}^d}u\psi_l\left| \frac{m_1}{m_1-1}\nabla u^{m_1-1}-\nabla v\right|^2\\
					&-(1-\eta)\int^t_0 \int_{\mathbb{R}^d}w\psi_l\left| \frac{m_2}{m_2-1}\nabla w^{m_2-1}-\nabla z\right|^2\\
					&+C\int^t_0\int_{\mathbb{R}^d}(u_0+w_0)(1-\psi_l)+\frac{CT}{\eta l^2}+\frac{CT}{l}\,\,\,\text{for}\,\,t\in(0,T)
				\end{split}
			\end{equation*}
			by the claimed convergence in Lemma \ref{global weak solution exist} and a lower semi-continuity of the free energy dissipation. Finally as $l\rightarrow +\infty$ and $\eta\rightarrow 0$,
			\begin{equation*}\label{3.28}
				\begin{split}
					\mathcal{F}[u(t),w(t)]\leq&\mathcal{F}[u_{0},w_{0}]
					-\int^t_0\int_{\mathbb{R}^d}u\left| \frac{m_1}{m_1-1}\nabla u^{m_1-1}-\nabla v\right|^2\\
					&-\int^t_0 \int_{\mathbb{R}^d}w\left| \frac{m_2}{m_2-1}\nabla w^{m_2-1}-\nabla z\right|^2\,\,\,\text{for}\,\,t\in(0,T).
				\end{split}
			\end{equation*}
			Therefore, $(u,w)$ is a free energy solution by the definition.
			
		\end{proof}
		
		\section{The free energy functional}
		
		Now we concentrate on a deeper analysis of the energy functional $\mathcal{F}$ given by
		$$
		\mathcal{F}[u(t),w(t)]
		=\frac{1}{m_1-1}\int_ {\mathbb{R}^d}u^{m_1}dx+\frac{1}{m_2-1}\int_ {\mathbb{R}^d}w^{m_2}dx
		-c_d\mathcal{H}[u,w]
		$$
		with decay property $\mathcal{F}[u(t),w(t)]\leq \mathcal{F}[u_0,w_0]$ for $t\geq 0$, where	
		\begin{equation*}
			\begin{split}
				\mathcal{H}[u,w]=\iint_{\mathbb{R}^d\times \mathbb{R}^d}\frac{u(x)w(y)}{|x-y|^{d-2}}dxdy=\int_ {\mathbb{R}^d}u(x)I_2(w)(x)dx
				=\int_ {\mathbb{R}^d}w(y)I_2(u)(y)dy.
			\end{split}
		\end{equation*}
		The estimate for $\mathcal{H}$ can be given as follows.
		\begin{lemma}{\label{estimate for H lem}}
			Let $\eta>0$, and let $m_1,m_2, m>1$.
			If
			\be\label{connection between m_1 and m_21}
			m<d/2\,\,\,\,\text{and}\,\,\,\,	mm_2+2mm_2/d\geq m+m_2,
			\ee
			then  for any $f\in L^{m}(\mathbb{R}^d)$ and $g\in L^1(\mathbb{R}^d)\cap L^{m_2}(\mathbb{R}^d)$, there holds
			\be{\label{Young inequality for H1}}
			\begin{split}
				 |\mathcal{H}[f,g]|&\leq\eta\|f\|^{m}_{{m}}+C\eta^{-\frac{1}{m-1}}\|g\|^{\frac{mm_2+2mm_2/d-m-m_2}{(m-1)(m_2-1)}}_1\|g\|^{\frac{m_2-2mm_2/d}{(m-1)(m_2-1)}}_{m_2}.
			\end{split}
			\ee
			Moreover, if
			\be\label{connection between m_1 and m_22}
			m<d/2\,\,\,\,\text{and}\,\,\,\,	mm_1+2mm_1/d\geq m+m_1,
			\ee
			then for any $f\in L^1(\mathbb{R}^d)\cap L^{m_1}(\mathbb{R}^d)$ and $g\in L^{m}(\mathbb{R}^d)$, there holds
			\be{\label{Young inequality for H22}}
			\begin{split}
				|\mathcal{H}[f,g]|
				&\leq C\eta^{-\frac{1}{m-1}}\|f\|^{\frac{m_1m+2m_1m/d-m_1-m}{(m_1-1)(m-1)}}_1\|f\|^{\frac{m_1-2m_1m/d}{(m_1-1)(m-1)}}_{m_1}+\eta\|g\|^{m}_{{m}}.
			\end{split}
			\ee
		\end{lemma}
		
		\begin{proof}
			Fixing $m\in\left(1,d/2\right)$, using H\"{o}lder's inequality with $\frac{1}{m}+\frac{m-1}{m}=1$ and the HLS inequality with $\lambda=2$  in Lemma {\ref{HLS inequality}},  we find that
			\begin{equation} {\label{inequality for H(f,g)}}
				\begin{split}
					\mathcal{H}[f,g]&=\int_ {\mathbb{R}^d}f(x)I_{2}(g)(x)dx\leq
					\|f\|_{m}\|I_{2}(g)\|_{\frac{m}{m-1}}\leq C_{\text{HLS}}\|f\|_{m}\|g\|_{\frac{md}{(d+2)m-d}}.
				\end{split}
			\end{equation}
			Since the assumption $m+m_2\leq mm_2+2mm_2/d$	
			ensures that
			\begin{equation*}
				\begin{split}
					1< \frac{md}{(d+2)m-d}\leq m_2,
				\end{split}
			\end{equation*}
			then if $g\in L^1(\mathbb{R}^d)\cap L^{m_2}(\mathbb{R}^d)$ with $m_2> 1$, the following interpolation inequality holds:
			\begin{equation*}
				\begin{split}
					\|g\|_{\frac{md}{(d+2)m-d}}\leq \|g\|^{\theta_1}_1\|g\|^{1-\theta_1}_{m_2}
				\end{split}
			\end{equation*}
			with $\frac{(d+2)m-d}{md}=\theta_1+\frac{1-\theta_1}{m_2}$, $\theta_1\in(0,1)$.	Hence
			\begin{equation*} {\label{property of H}}
				\begin{split}
					|\mathcal{H}[f,g]|&\leq C_{\text{HLS}}\|f\|_{{m}}\|g\|^{\theta_1}_1\|g\|^{1-\theta_1}_{m_2}\\
					 &\leq\eta\|f\|^{m}_{{m}}+C\eta^{-\frac{1}{m-1}}\|g\|^{\frac{mm_2+2mm_2/d-m-m_2}{(m-1)(m_2-1)}}_1\|g\|^{\frac{m_2-2mm_2/d}{(m-1)(m_2-1)}}_{m_2},
				\end{split}
			\end{equation*}
			by Young's inequality, which implies (\ref{Young inequality for H1}). (\ref{Young inequality for H22}) can be also proved if (\ref{connection between m_1 and m_22}) holds.
		\end{proof}

		We establish several variants to the HLS inequality on the lines $L_1, L_2$ and the intersection point $\bf{I}$.
		\begin{lemma}{\label{estimate for H lem2}}
			Let  ${\bf{m}}$ be on $L_1$ , and let  $f\in L^{m_1}(\mathbb{R}^d)$ and $g\in L^1(\mathbb{R}^d)\cap L^{m_2}(\mathbb{R}^d)$. Then
			$$
			C_*:=\sup_{f\neq 0,g\neq 0}\left\{\frac{|\mathcal{H}[f,g]|}{\|f\|_{m_1}\|g\|^{2/d}_{1}\|g\|^{1-2/d}_{m_2}}\right\}<\infty.
			$$
			If ${\bf{m}}$ is on $L_2$, and $f\in L^1(\mathbb{R}^d)\cap L^{m_1}(\mathbb{R}^d)$ and $g\in L^{m_2}(\mathbb{R}^d)$, then
			$$
			C_{\star}:=\sup_{f\neq 0,g\neq 0}\left\{\frac{|\mathcal{H}[f,g]|}{\|f\|^{2/d}_{1}\|f\|^{1-2/d}_{m_1}\|g\|_{m_2}}\right\}<\infty.
			$$
			In addtion, assume that ${\bf{m}}$ is ${\bf{I}}$ and $(f,g)\in \left(L^1(\mathbb{R}^d)\cap L^{m_c}(\mathbb{R}^d)\right)^2$. Then
			\be{\label{upper for critical number}}
			\begin{split}
				C_{c}:=\sup_{f\neq 0,g\neq 0}\left\{\frac{\mathcal{H}[f,g]}{\|f\|^{1/d}_{1}\|f\|^{m_c/2}_{m_c}\|g\|^{1/d}_{1}\|g\|^{m_c/2}_{m_c}}\right\}<\infty.
			\end{split}
			\ee
		\end{lemma}
		\begin{proof}
			If ${\bf{m}}$ is on $L_1$, then $m_1\in(m_c,d/2)$ and using (\ref{inequality for H(f,g)}) with $m=m_1$ we have
			\begin{equation*}
				\begin{split}
					|\mathcal{H}[f,g]|\leq  C_{\text{HLS}}\|f\|_{m_1}\|g\|_{\frac{m_1d}{(d+2)m_1-d}}\leq C_{\text{HLS}}\|f\|_{m_1} \|g\|^{\frac{2}{d}}_{1}\|g\|^{1-\frac{2}{d}}_{m_2}.
				\end{split}
			\end{equation*}
			Therefore, $C_*$  is finite and bounded above by $C_{\text{HLS}}$. It is also easy to see that $C_{\star}$ is controlled by $C_{\text{HLS}}$ if ${\bf{m}}$ is on $L_2$. Finally, with the help of the HLS inequality and H\"{o}lder's inequality, we find that
			\begin{equation*}
				\begin{split}
					|\mathcal{H}[f,g]|&\leq C_{\text{HLS}}\|f\|_{\frac{2d}{d+2}}\|g\|_{\frac{2d}{d+2}}\leq C_{\text{HLS}}\|f\|^{1/d}_{1}\|f\|^{m_c/2}_{m_c}\|g\|^{1/d}_{1}\|g\|^{m_c/2}_{m_c}
				\end{split}
			\end{equation*}
			if ${\bf{m}}$ is ${\bf{I}}$. Then the definition of $C_{c}$ is valid.
		\end{proof}

		Define
		$$
		M_{1c}=(c_dC_{\star})^{-d/2}\left(m_2/(m_2-1)\right)^{d/2}(m_1-1)^{-d(m_2-1)/(2m_2)},
		$$
		$$
		M_{2c}=(c_dC_{*})^{-d/2}\left(m_1/(m_1-1)\right)^{d/2}(m_2-1)^{-d(m_1-1)/(2m_1)},
		$$
		and
		$$
		M_{c}=\left(2/[c_dC_c(m_c-1)]\right)^{d/2}.
		$$
		The lower and upper bounds for $\mathcal{F}$ in the sets $S_{M_1}\times S_{M_2}$ below is given next.
		\begin{lemma}{\label{infimum for FEF4}}
			Let $(f,g)$ satisfy $f\in S_{M_1}$ and $g\in S_{M_2}$. If ${\bf{m}}$ is on $L_1$, then
			\be\label{infimum inequlity for F in case 1}
			\begin{split}
				 &\left(c_dC_{*}\right)^{\frac{m_1}{m_1-1}}(m_1-1)^{\frac{m_1}{m_1-1}}m^{-\frac{m_1}{m_1-1}}_1\left(M^{\frac{2m_1}{d(m_1-1)}}_{2c}-M^{\frac{2m_1}{d(m_1-1)}}_{2}\right)\|g\|^{m_2}_{m_2}\\
				&\leq\mathcal{F}[f,g]
				\leq\frac{2}{m_1-1}\|f\|^{m_1}_{m_1}\\
				 &+\left(c_dC_{*}\right)^{\frac{m_1}{m_1-1}}(m_1-1)^{\frac{m_1}{m_1-1}}m^{-\frac{m_1}{m_1-1}}_1\left(M^{\frac{2m_1}{d(m_1-1)}}_{2c}+M^{\frac{2m_1}{d(m_1-1)}}_{2}\right)\|g\|^{m_2}_{m_2}
			\end{split}
			\ee
			and
			$$
			\inf_{f\in S_{M_1}}\inf_{g\in S_{M_2}}\mathcal{F}[f,g]=0,\,\,\text{if}\,\,M_2\in(0,M_{2c}].
			$$
			If ${\bf{m}}$ is on $L_2$, then
			\be\label{infimum inequlity for F in case 2}
			\begin{split}
				\mathcal{F}[f,g]\geq \left(c_dC_{\star}\right)^{\frac{m_2}{m_2-1}}(m_2-1)^{\frac{m_2}{m_2-1}}m^{-\frac{m_2}{m_2-1}}_2\left(M^{\frac{2m_2}{d(m_2-1)}}_{1c}-M^{\frac{2m_2}{d(m_2-1)}}_{1}\right)\|f\|^{m_1}_{m_1}
			\end{split}
			\ee
			and
			\be\label{infimum for F in case 2}
			\begin{split}
				\inf_{f\in S_{M_1}}\inf_{g\in S_{M_2}}\mathcal{F}[f,g]=0,\,\,&\text{if}\,\,M_1\in(0,M_{1c}].
			\end{split}
			\ee
			If ${\bf{m}}$ is ${\bf{I}}$, then
			$$
			 \mathcal{F}[f,g]\geq\frac{(c_dC_c)^2(m_c-1)}{4}\left(M^{\frac{4}{d}}_c-M^{\frac{2}{d}}_1M^{\frac{2}{d}}_2\right)\|g\|^{m_c}_{m_c}
			$$
			or
			$$
			 \mathcal{F}[f,g]\geq\frac{(c_dC_c)^2(m_c-1)}{4}\left(M^{\frac{4}{d}}_c-M^{\frac{2}{d}}_1M^{\frac{2}{d}}_2\right)\|f\|^{m_c}_{m_c}.
			$$
			Furthermore,
			\be\label{infimum for F in case 3}
			\begin{split}
				\inf_{f\in S_{M_1}}\inf_{g\in S_{M_2}}\mathcal{F}[f,g]=0,\,\,&\text{if}\,\,M_1M_2\in(0,M^2_{c}].
			\end{split}
			\ee
		\end{lemma}
		\begin{proof}
			By Lemma \ref{estimate for H lem2}, $\mathcal{H}$ satisfies
			\begin{equation*}	
				\begin{split}
					|\mathcal{H}[f,g]|\leq& C_{*}\|f\|_{m_1}\|g\|^{\frac{2}{d}}_{1}\|g\|^{1-\frac{2}{d}}_{m_2}\\
					\leq& \frac{1}{c_d(m_1-1)}\|f\|^{m_1}_{m_1}\\
					 &+C_{*}\left(c_dC_{*}\right)^{\frac{1}{m_1-1}}\left(\frac{m_1-1}{m_1}\right)^{\frac{m_1}{m_1-1}}\|g\|^{\frac{2m_1}{d(m_1-1)}}_{1}\|g\|^{\left(1-\frac{2}{d}\right)\frac{m_1}{m_1-1}}_{m_2}\\
					=&\frac{1}{c_d(m_1-1)}\|f\|^{m_1}_{m_1}\\
					 &+C_{*}\left(c_dC_{*}\right)^{\frac{1}{m_1-1}}\left(\frac{m_1-1}{m_1}\right)^{\frac{m_1}{m_1-1}}\|g\|^{\frac{2m_1}{d(m_1-1)}}_{1}\|g\|^{m_2}_{m_2}.
				\end{split}
			\end{equation*}	
			Then $\mathcal{F}$ can be estimated as
			\begin{equation*}
				\begin{split}
					 \mathcal{F}[f,g]=&\frac{1}{m_1-1}\|f\|^{m_1}_{m_1}+\frac{1}{m_2-1}\|g\|^{m_2}_{m_2}-c_d\mathcal{H}[f,g]\\
					\geq&\frac{1}{m_2-1}\|g\|^{m_2}_{m_2}\\
					 &-\left(c_dC_{*}\right)^{\frac{m_1}{m_1-1}}\left(\frac{m_1-1}{m_1}\right)^{\frac{m_1}{m_1-1}}\|g\|^{\frac{2m_1}{d(m_1-1)}}_{1}\|g\|^{m_2}_{m_2}\\
					 =&\left(c_dC_{*}\right)^{\frac{m_1}{m_1-1}}\left(\frac{m_1-1}{m_1}\right)^{\frac{m_1}{m_1-1}}\left(M^{\frac{2m_1}{d(m_1-1)}}_{2c}-M^{\frac{2m_1}{d(m_1-1)}}_{2}\right)\|g\|^{m_2}_{m_2}
				\end{split}
			\end{equation*}
			and
			\begin{equation*}
				\begin{split}
					\mathcal{F}[f,g]	 \leq&\frac{2}{m_1-1}\|f\|^{m_1}_{m_1}\\
					 &+\left(c_dC_{*}\right)^{\frac{m_1}{m_1-1}}\left(\frac{m_1-1}{m_1}\right)^{\frac{m_1}{m_1-1}}\left(M^{\frac{2m_1}{d(m_1-1)}}_{2c}+M^{\frac{2m_1}{d(m_1-1)}}_{2}\right)\|g\|^{m_2}_{m_2}.
				\end{split}
			\end{equation*}
			In the case $M_2\leq M_{2c}$, since $\mathcal{F}\geq 0$, then the infimum is nonnegative. Taking
			\begin{equation*}
				\begin{split}
					h_1(x,t)=\frac{M_1}{(4\pi t)^{\frac{d}{2}}}e^{-\frac{|x|^2}{4t}}\,\,\,\text{and}\,\,\,h_2(x,t)=\frac{M_2}{(4\pi t)^{\frac{d}{2}}}e^{-\frac{|x|^2}{4t}},
				\end{split}
			\end{equation*}
			it is obvious that $h_i\in L^1(\mathbb{R}^d)$ with $\|h_i\|_1=M_i$, $i=1,2$, satisfy
			\begin{equation*}
				\begin{split}
					\|h_i\|^{m_i}_{m_i}=O(t^{-\frac{d(m_i-1)}{2}}),
				\end{split}
			\end{equation*}
			which implies that $h_i\in S_{M_i}$ and that $\mathcal{F}[h_1,h_2]$ tends to 0 as $t\rightarrow\infty$. Therefore,
			$$\inf_{f\in S_{M_1}}\inf_{g\in S_{M_2}}\mathcal{F}[f,g]=0.$$
			
			If ${\bf{m}}$ is on $L_2$, we have (\ref{infimum inequlity for F in case 2}) one more by the HLS inequality and H\"{o}lder's inequality, and take $h_i$ above to see (\ref{infimum for F in case 2}).
			
			If ${\bf{m}}$ is ${\bf{I}}$, since
			\begin{equation*}	
				\begin{split}
					|\mathcal{H}[f,g]|\leq&  C_cM^{\frac{1}{d}}_{1}M^{\frac{1}{d}}_{2}\|f\|^{\frac{m_c}{2}}_{m_c}\|g\|^{\frac{m_c}{2}}_{m_c}
					\leq \frac{1}{c_d(m_c-1)}\|f\|^{m_c}_{m_c}
					+\frac{M^{\frac{2}{d}}_1M^{\frac{2}{d}}_2}{c_d(m_c-1)M^{\frac{4}{d}}_c} \|g\|^{m_c}_{m_c}
				\end{split}
			\end{equation*}
			or
			\begin{equation*}	
				\begin{split}
					|\mathcal{H}[f,g]|\leq&
					\frac{M^{\frac{2}{d}}_1M^{\frac{2}{d}}_2}{c_d(m_c-1)M^{\frac{4}{d}}_c} \|f\|^{m_c}_{m_c}
					+\frac{1}{c_d(m_c-1)} \|g\|^{m_c}_{m_c}
				\end{split}
			\end{equation*}
			by Young's inequality, then $\mathcal{F}$ satisfies
			\begin{equation*}
				\begin{split}
					\mathcal{F}[f,g]
					 \geq&\frac{1}{m_c-1}\|f\|^{m_c}_{m_c}+\frac{1}{m_c-1}\|g\|^{m_c}_{m_c}-c_dC_cM^{\frac{1}{d}}_{1}M^{\frac{1}{d}}_{2}\|f\|^{\frac{m_c}{2}}_{m_c}\|g\|^{\frac{m_c}{2}}_{m_c}\\
					 \geq&\frac{(c_dC_c)^2(m_c-1)}{4}\left(\frac{4}{(c_dC_c(m_c-1))^2}-M^{\frac{2}{d}}_{1}M^{\frac{2}{d}}_{2}\right)\|g\|^{m_c}_{m_c}
				\end{split}
			\end{equation*}
			or
			\begin{equation*}
				\begin{split}
					\mathcal{F}[f,g]
					 \geq&\frac{(c_dC_c)^2(m_c-1)}{4}\left(\frac{4}{(c_dC_c(m_c-1))^2}-M^{\frac{2}{d}}_{1}M^{\frac{2}{d}}_{2}\right)\|f\|^{m_c}_{m_c}.
				\end{split}
			\end{equation*}
			One finally obtains from
			\begin{equation*}
				\begin{split}
					\mathcal{F}[f,g]	 \leq&\frac{2}{m_c-1}\|f\|^{m_1}_{m_1}
					 +\frac{(c_dC_c)^2(m_c-1)}{4}\left(\frac{4}{(c_dC_c(m_c-1))^2}+M^{\frac{2}{d}}_{1}M^{\frac{2}{d}}_{2}\right)\|g\|^{m_c}_{m_c}
				\end{split}
			\end{equation*}
			that (\ref{infimum for F in case 3}) is true by taking $f=h_1$ and $g=h_2$.
		\end{proof}
		
		The characterization of non-zero minimizers of  $\mathcal{F}$ in $S_{M_1}\times S_{M_2}$ on critical lines and point is the goal in this subsection. If  ${\bf{m}}$ is  ${\bf{I}}$, the existence of global minimizers is guaranteed in particular situation. The proof is inspired by  {\cite[Proposition 3.5]{Carrillo09-CVPDE}}.
		
		\begin{theorem}{\label{extremals of H in case 3}}
			Let ${\bf{m}}$  be ${\bf{I}}$. Then there exist a pair of nonnegative, radially symmetric and non-increasing functions $(f^*,g^*)\in \left(L^1(\mathbb{R}^d)\cap  L^{m_c}(\mathbb{R}^d)\right)^2$ such that
			$$
			\mathcal{H}[f^*,g^*]=C_{c}.
			$$
In addition, there exists a minimizer $(f,g)\in  S_{M_1}\times S_{M_2}$ of $\mathcal{F}$ if $M_1=M_2=M_c$, and the minimizer satisfies
			$$
			f(x)=g(x)=\begin{cases}
				\frac{1}{R^d_0}\left[\zeta\left(\frac{x-x_0}{R_0}\right)\right]^{d/(d-2)}
				,\,\,\,&\text{if}\,\,\,x\in B\left(x_0,R_0\right), \\[0.2cm]
				0,\,\,\,&\text{if}\,\,\,x\in \mathbb{R}^d\backslash B\left(x_0,R_0\right)
			\end{cases}
			$$
			with some $R_0>0$  and $x_0\in\mathbb{R}^d$, where $\zeta$ is the unique positive radial classical solution to the Lane-Emden equation
			$$
			\begin{cases}
				-\Delta\zeta=\frac{m_c-1}{m_c}\zeta^{1/(m_c-1)}, &x\in B(0,1),\,\, \\[0.2cm]
				\zeta=0,\,\,&x\in\partial B(0,1).
			\end{cases}
			$$
		\end{theorem}
		\begin{proof}
			We claim that if $C_c$ in (\ref{upper for critical number}) is obtained by some non-zero $f$ and $g$, then $g=c_0f$ with some $c_0$. This is easily verified by the positive definite of $|x-y|^{-(d-2)}$, see {\cite[Theorem 9.8]{Lieb2001}}. In fact, suppose that there exist a pair of maximizing nonnegative functions $(f,g)\in \left(L^1(\mathbb{R}^d)\cap  L^{m_c}(\mathbb{R}^d)\right)^2$ such that
			\begin{equation*}
				\begin{split}
					 \mathcal{H}[f,g]=C_c\|f\|^{\frac{1}{d}}_{1}\|f\|^{\frac{m_c}{2}}_{m_c}\|g\|^{\frac{1}{d}}_{1}\|g\|^{\frac{m_c}{2}}_{m_c}.
				\end{split}
			\end{equation*}
			Then by {\cite[Theorem 9.8]{Lieb2001}} and the HLS inequality,
			\be\label{equal in the case 3 2}
			\begin{split}
				\mathcal{H}[f,g]\leq& \sqrt{\mathcal{H}[f,f]}\cdot\sqrt{\mathcal{H}[g,g]}\\
				 \leq&C_{c}\|f\|^{\frac{1}{d}}_{1}\|f\|^{\frac{m_c}{2}}_{m_c}\|g\|^{\frac{1}{d}}_{1}\|g\|^{\frac{m_c}{2}}_{m_c}.
			\end{split}
			\ee
			However, (\ref{equal in the case 3 2}) is an equality if and only if $g=c_0f$ with some constant $c_0$.
			
			Note that	
			\be{\label{upper for critical number in equal case}}
			\begin{split}
				C_{c}=\sup_{f\neq 0}\left\{\frac{\mathcal{H}[f,f]}{\|f\|^{\frac{2}{d}}_{1}\|f\|^{m_c}_{m_c}},\,\,\,f\in L^1(\mathbb{R}^d)\cap L^{m_c}(\mathbb{R}^d)\right\}.
			\end{split}
			\ee
			The existence of a maximizing nonnegative, radially symmetric and non-increasing $f^*$ with $\|f^*\|_{1}=\|f^*\|_{m_c}=1$ for (\ref{upper for critical number in equal case}) has been given in {\cite[Proposition 3.3]{Carrillo09-CVPDE}}. So choosing $g^*=c_0f^*$, then
			$\mathcal{H}[f^*,g^*]=c_0C_c$ and the first conclusion has been proved with $c_0=1$.
			
			To derive minimizers for $\mathcal{F}$ in the situation $M_1=M_2=M_c$, with
			$f:=M_cf^*$ and  $g:=M_cf^*$ we  have $(f,g)\in S_{M_1}\times S_{M_2}$ with $\|f\|_{1}=\|f\|_{m_c}=M_c, \|g\|_{1}=\|g\|_{m_c}=M_c$. After a careful computation we infers that
			\begin{equation*}	
				\begin{split}
					\mathcal{F}[f,g]=0
				\end{split}
			\end{equation*}	
			by the definition of $M_c$ and $(f,g)$ is a non-zero global minimizer of $\mathcal{F}$  in $S_{M_1}\times S_{M_2}$.
			The precisely description of the set of minimizers of $\mathcal{F}$ was derived in {\cite[Proposition 3.5]{Carrillo09-CVPDE}}, we omit it here and have proved the second conclusion.	
			
		\end{proof}
		
		On $L_1$, we assert that there is no non-zero minimizer of $\mathcal{F}$ in $S_{M_1}\times S_{M_2}$  if $M_2=M_{2c}$. The proof includes two steps: the first one is to derive the nonexistence of non-trivial classical solution to a  Lane-Emden system (see Lemma \ref{nonexistence for LE system in case 1}), and the second is to make a contradiction by the achievement of Euler-Largrange equalities which consist of the Lane-Emden system on the assumption that minimizers of its free energy exist  (see Theorem \ref{nonexistence of minimiser in the case 1} ).
		\begin{lemma}{\label{nonexistence for LE system in case 1}}
			Let $M_1,M_2,\rho>0$, and let $m_1>1$ and $m_2>1$. Consider a Lane-Emden system
			\be\label{Lane-Emden system}
			\begin{split}
				\begin{cases}
					&-\Delta\vartheta(x)=\frac{m_1-1}{m_1}\varsigma^{\frac{1}{m_2-1}}(x),\,\,\,\,x\in \Omega_1=\mathbb{R}^d,\\
					&-\Delta\varsigma(x)=\frac{m_2-1}{m_2}\vartheta^{\frac{1}{m_1-1}}(x),\,\,\,\,x\in \Omega_2=B(0,\rho),\\
					&\varsigma(x)=0,\,\,\,\,x\in \mathbb{R}^d\setminus \Omega_2.
				\end{cases}
			\end{split}
			\ee
			Then (\ref{Lane-Emden system}) does not admit any nonnegative and non-trivial classical solution $(\vartheta,\varsigma)\in \left(L^{1/(m_1-1)}(\mathbb{R}^d)\cap L^{m_1/(m_1-1)}(\mathbb{R}^d)\right)\times \left(L^{1/(m_2-1)}(\mathbb{R}^d)\cap L^{m_2/(m_2-1)}(\mathbb{R}^d)\right)$ with \\		
			$\|\vartheta^{1/(m_1-1)}\|_1=M_1$ and $\|\varsigma^{1/(m_2-1)}\|_1=M_2$, provided that ${\bf{m}}$ is on $L_1$.
		\end{lemma}
		\begin{proof}
			Let
			\begin{equation*}
				\begin{split}
					q:=\frac{1}{m_1-1}\in \left(\frac{2}{d-2},\frac{d}{d-2}\right).
				\end{split}
			\end{equation*}
			The existence/nonexistence of solutions to the general form  of Lane-Emden system has been investigated in \cite{Mitidieri1996,SZ1996,SZ1998}, for example. However, the solvability of (\ref{Lane-Emden system}) involving both whole space and bounded domains has not yet known as far as we know. We assert that there exists no non-trivial classical solution for (\ref{Lane-Emden system})  if  ${\bf{m}}$ is on $L_1$.
			
			Consider the following properties: Suppose that $\omega\in C^2(\mathbb{R}^d)$ is non-trivial and satisfies $\Delta w\leq 0,\,\,\,x\in\mathbb{R}^d$. Then
			\be\label{lower for omega}
			\begin{split}
				\omega(x)\geq C|x|^{2-d},\,\,\,\,\,|x|\geq 1
			\end{split}
			\ee
			by the strong maximum principle (see {\cite[Proposition 3.4]{SZ1996}}).  Relying on the finite of $\|\vartheta\|_q$, we have the following contradiction: For $R>1$,
			\begin{equation*}\label{contradiction for L1 norm of vartheta}
				\begin{split}
					M_1\geq \int_{B(0,R)}\vartheta^q=c_d\int^R_0\int_{\mathbb{S}^{d-1}}\vartheta^q(r,\theta)r^{d-1}dS(\theta)dr,
				\end{split}
			\end{equation*}
			where one combines with the fact that $\Delta \vartheta\leq 0 $ for $x\in \Omega_1=\mathbb{R}^d$ and  (\ref{lower for omega}) to see that
			\begin{equation*}\label{contradiction for L1 norm of vartheta2}
				\begin{split}
					M_1\geq&  C\int^R_1r^{d-1+q(2-d)}dr
					= C\int^R_1r^{\frac{dm_1+2-2d}{m_1-1}-1}dr\\
					 &=\frac{C(m_1-1)}{dm_1+2-2d}\left(R^{\frac{dm_1+2-2d}{m_1-1}}-1\right)\rightarrow\infty\,\,\,\text{as}\,\,\,R\rightarrow\infty
				\end{split}
			\end{equation*}
			due to $m_1>m_c=2-2/d$. So (\ref{Lane-Emden system}) has no non-trivial and nonnegative classical solution.\\
		\end{proof}

		\begin{theorem}{\label{nonexistence of minimiser in the case 1}}
			Let ${\bf{m}}$ be on $L_1$. For all $M_2\leq M_{2c}$, then  $\mathcal{F}$  does not admit any non-zero  minimizer in $S_{M_1}\times S_{M_2}$. 	
		\end{theorem}
		\begin{proof}
			The left inequality in (\ref{infimum inequlity for F in case 1}) in Lemma {\ref{infimum for FEF4}} makes sure that there exists no minimizer if $M_2<M_{2c}$. Thus we only consider $M_2=M_{2c}$ and  prove it by contradiction.\\
			{\bf{Step 1.}}  \emph{Necessary conditions for global minimizers of  $\mathcal{F}$}. We assume that  minimizers exist and try to present some basic properties of them. Suppose that $(f^*,g^*)\in S_{M_1}\times S_{M_2}$ is a minimizer of  $\mathcal{F}$ in the sense that $\mathcal{F}[f^*,g^*]=0$. Then
			\begin{equation}\label{inequality for H3}
				\begin{split}
					&\frac{1}{m_1-1}\|f^*\|^{m_1}_{m_1}+\frac{1}{m_2-1}\|g^*\|^{m_2}_{m_2}
					=c_d\mathcal{H}[f^*,g^*]\\
					&\leq c_dC_*\|f^*\|_{m_1}\|g^*\|^{2/d}_{1}\|g^*\|^{1-2/d}_{m_2}\\
					&\leq \frac{1}{m_1-1}\|f^*\|^{m_1}_{m_1}
					+\left(c_dC_{*}\right)^{\frac{m_1}{m_1-1}}\left(\frac{m_1-1}{m_1}\right)^{\frac{m_1}{m_1-1}}
					\|g^*\|^{\frac{2m_1}{d(m_1-1)}}_1\|g^*\|^{\left(1-\frac{2}{d}\right)\frac{m_1}{m_1-1}}_{m_2}\\
					&= \frac{1}{m_1-1}\|f^*\|^{m_1}_{m_1}+\frac{1}{m_2-1}M^{-\frac{2m_1}{d(m_1-1)}}_{2c}
					\|g^*\|^{\frac{2m_1}{d(m_1-1)}}_1\|g^*\|^{m_2}_{m_2}\\
					&=
					\frac{1}{m_1-1}\|f^*\|^{m_1}_{m_1}+\frac{1}{m_2-1}M^{-\frac{2m_1}{d(m_1-1)}}_{2c}
					M^{\frac{2m_1}{d(m_1-1)}}_2\|g^*\|^{m_2}_{m_2}\\
					&=\frac{1}{m_1-1}\|f^*\|^{m_1}_{m_1}+\frac{1}{m_2-1}\|g^*\|^{m_2}_{m_2}
				\end{split}
			\end{equation}
			by the HLS inequality, Young's inequality,\ the definition of $M_{2c}$ and $M_2=M_{2c}$. As a consequence of  (\ref{inequality for H3}), we obtain that
			\begin{equation}\label{relation between u and w in the case of minimiser}
				\begin{split}
					\|f^*\|^{m_1}_{m_1}=&\frac{1}{m_2-1}M^{-\frac{2m_1}{d(m_1-1)}}_{2c}
					\|g^*\|^{\frac{2m_1}{d(m_1-1)}}_1\|g^*\|^{m_2}_{m_2}\\
					=&\frac{1}{m_2-1}M^{-\frac{2m_1}{d(m_1-1)}}_{2c}M^{\frac{2m_1}{d(m_1-1)}}_2\|g^*\|^{m_2}_{m_2}\\
					=&\frac{1}{m_2-1}\|g^*\|^{m_2}_{m_2}
				\end{split}
			\end{equation}
			and
			\begin{equation*}\label{H in the case of minimiser}
				\begin{split}
					\mathcal{H}[f^*,g^*]
					 &=C_*\|f^*\|_{m_1}\|g^*\|^{2/d}_{1}\|g^*\|^{1-2/d}_{m_2}=\frac{m_1}{c_d(m_1-1)(m_2-1)}\|g^*\|^{m_2}_{m_2}.
				\end{split}
			\end{equation*}
			{\bf{Step 2.}}  \emph{The Euler-Lagrange equalities}. Let $f$ and $g$ be symmetric rearrangement of $f^*$ and $g^*$. Then  $(f,g)\in S_{M_1}\times S_{M_2}$ satisfies
			\begin{equation}\label{relation betwee f and g in the case minimiser}
				\begin{split}
					 \|f\|^{m_1}_{m_1}=\|f^*\|^{m_1}_{m_1}=\frac{1}{m_2-1}\|g^*\|^{m_2}_{m_2}=\frac{1}{m_2-1}\|g\|^{m_2}_{m_2}
				\end{split}
			\end{equation}
			and
			$$
			\mathcal{H}[f,g]\geq\mathcal{H}[f^*,g^*]
			$$
			by (\ref{relation between u and w in the case of minimiser}) and the Riesz rearrangement properties {\cite[Lemma 2.1]{Lieb83-CMP}}.  Obviously, $\mathcal{F}[f,g]=0$ and $(f,g)$ is also a minimizer of $\mathcal{F}$. Note that
			\begin{equation}\label{H in the case of minimiser2}
				\begin{split}
					c_d\mathcal{H}[f,g]
					&=\frac{m_1}{m_1-1}\|f\|^{m_1}_{m_1}=\frac{m_1}{(m_1-1)(m_2-1)}\|g\|^{m_2}_{m_2}.
				\end{split}
			\end{equation}

			Given $\Omega_{10}=\{x\in\mathbb{R}^d:f(x)=0\}$ and $\Omega_{1+}=\{x\in\mathbb{R}^d:f(x)>0\}$ and introduce $\phi_1\in C^{\infty}_0(\mathbb{R}^d)$ with $\phi_1(x)=\phi_1(-x)$ and
			\begin{equation*}
				\begin{split}	 \psi_1(x)=\frac{f(x)}{M_1}\left(\phi_1(x)-\frac{1}{M_1}\int_{\mathbb{R}^d}f(x)\phi_1(x)dx\right).
				\end{split}
			\end{equation*}
			
			Then for $f\in S_{M_1}$ and fix $\epsilon\in(0,\epsilon_0:=M_1(2\|\phi_1\|_\infty)^{-1})$, there holds
			\begin{equation*}
				\begin{split}
					\|f+\epsilon \psi_1\|_1=M_1
				\end{split}
			\end{equation*}
			and
			\begin{equation*}
				\begin{split}
					f+\epsilon \psi_1=&f\left(1+\frac{\epsilon}{M_1}\left(\phi_1(x)-\frac{1}{M_1}\int_{\mathbb{R}^d}f(x)\phi_1(x)dx
					\right)\right)\\
					\geq& f\left(1-\frac{2\|\phi_1\|_{\infty}\epsilon}{M_1}\right)\geq 0,
				\end{split}
			\end{equation*}
			which implies that $f+\epsilon \psi_1\in S_{M_1}$.  Moreover, supp $(\psi_1)\subset \overline{\Omega}_{1+}$.	 
			Then
			\begin{equation*}\label{variation of f}
				\begin{split}
					\frac{\mathcal{F}[f+\epsilon \psi_1,g]-\mathcal{F}[f,g]}{\epsilon}=\frac{1}{m_1-1}\int_{\Omega_{1+}}	 \frac{(f+\epsilon\psi_1)^{m_1}-f^{m_1}}{\epsilon}-
					\int_{\mathbb{R}^d}\mathcal{K}\ast g(x)\psi_1(x)dx.
				\end{split}
			\end{equation*}
			According to $\mathcal{F}[f+\epsilon\psi_1, g]\geq \mathcal{F}[f,g]$,
			as $\epsilon\rightarrow 0$, Lebesgue's dominated convergence theorem shows that
			\begin{equation*}
				\begin{split}
					\int_{\mathbb{R}^d}	 \left(\frac{m_1}{m_1-1}f^{m_1-1}(x)-\mathcal{K}\ast g(x)\right)\psi_1(x)dx\geq0.
				\end{split}
			\end{equation*}	
			By replacing $-\psi_1$ by $\psi_1$, one also obtains from above to see that
			\begin{equation*}
				\begin{split}
					\int_{\mathbb{R}^d}	 \left(\frac{m_1}{m_1-1}f^{m_1-1}(x)-\mathcal{K}\ast g(x)\right)\psi_1(x)dx=0,
				\end{split}
			\end{equation*}
			where
			\begin{equation*}
				\begin{split} 0=&\frac{1}{M_1}\int_{\mathbb{R}^d}\left(\frac{m_1}{m_1-1}f^{m_1-1}(x)-\mathcal{K}\ast g(x)\right)f(x)\phi_1(x)dx\\
					&-\frac{1}{M^2_1}\int_{\mathbb{R}^d}f(x)\phi_1(x)dx\cdot\int_{\mathbb{R}^d}
					\left(\frac{m_1}{m_1-1}f^{m_1}(x)-\mathcal{K}\ast f(x)g(x)\right)dx\\
					=&\frac{1}{M_1}\int_{\mathbb{R}^d}\left(\frac{m_1}{m_1-1}f^{m_1-1}(x)-\mathcal{K}\ast g(x)\right)f(x)\phi_1(x)dx
				\end{split}
			\end{equation*}
			by (\ref{H in the case of minimiser2}). For any choice of symmetric test function $\phi_1\in C^{\infty}_0(\mathbb{R}^d)$, we also obtain
			\begin{equation*}
				\begin{split}{\label{equal for f in whole domain2}}
					\frac{m_1}{m_1-1}f^{m_1-1}(x)-\mathcal{K}\ast g(x)= 0\,\,\,a.e. \,\,\,\text{in} \,\,\,\mathbb{R}^d.
				\end{split}
			\end{equation*}
			
			For $g$, arguing similarly as above and we define
			$\Omega_{20}=\{x\in\mathbb{R}^d:g(x)=0\}$ and $\Omega_{2+}=\{x\in\mathbb{R}^d:g(x)>0\}$ and introduce $\phi_2\in C^{\infty}_0(\mathbb{R}^d)$ with $\phi_2(x)=\phi_2(-x)$ and \begin{equation*}
				\begin{split}	 \psi_2(x)=\frac{g(x)}{M_2}\left(\phi_2(x)-\frac{1}{M_2}\int_{\mathbb{R}^d}g(x)\phi_2(x)dx\right).
				\end{split}
			\end{equation*}
			Then for $g\in S_{M_2}$ and fix $\epsilon\in(0,M_2(2\|\phi_2\|_\infty)^{-1})$, there holds
			$g+\epsilon \psi_2\in S_{M_2}$. Then
			\begin{equation*}\label{variation of g}
				\begin{split}
					\frac{\mathcal{F}[f,g+\epsilon \psi_2]-\mathcal{F}[f,g]}{\epsilon}=&\frac{1}{m_2-1}\int_{\Omega_{2+}}	 \frac{(g+\epsilon\psi_2)^{m_2}-g^{m_2}}{\epsilon}dy\\
					&-\int_{\mathbb{R}^d}\mathcal{K}\ast f(y)\psi_2(y)dy,
				\end{split}
			\end{equation*}
			where by Lebesgue's dominated convergence theorem again and replacing $-\psi_2$ by $\psi_2$,  it follows that
			\begin{equation*}
				\begin{split}
					\int_{\mathbb{R}^d}	 \left(\frac{m_2}{m_2-1}g^{m_2-1}(y)-\mathcal{K}\ast f(y)\right)\psi_2(y)dy=0.
				\end{split}
			\end{equation*}
			Then (\ref{relation betwee f and g in the case minimiser})
			and (\ref{H in the case of minimiser2}) imply that
			\begin{equation*}
				\begin{split} 0=&\frac{1}{M_2}\int_{\mathbb{R}^d}\left(\frac{m_2}{m_2-1}g^{m_2-1}(y)-\mathcal{K}\ast f(y)\right)g(y)\phi_2(y)dy\\
					&-\frac{1}{M^2_2}\int_{\mathbb{R}^d}g(y)\phi_2(y)dy\cdot\int_{\mathbb{R}^d}
					\left(\frac{m_2}{m_2-1}g^{m_2}(y)-\mathcal{K}\ast f(y)g(y)\right)dy\\
					=&\frac{1}{M_2}\int_{\mathbb{R}^d}\left(\frac{m_2}{m_2-1}g^{m_2-1}(y)-\mathcal{K}\ast f(y)\right)g(y)\phi_2(y)dy\\
					&+\frac{2m_1}{M^2_2(d-2m_1)}\|g\|^{m_2}_{m_2}\int_{\mathbb{R}^d}g(y)\phi_2(y)dy\\
					=&\frac{1}{M_2}\int_{\mathbb{R}^d}\left(\frac{m_2}{m_2-1}g^{m_2-1}(y)-\mathcal{K}\ast f(y)
					+\frac{2m_1\|g\|^{m_2}_{m_2}}{M_2(d-2m_1)}\right)g(y)\phi_2(y)dy
				\end{split}
			\end{equation*}
			on $L_1$. Therefore,
			\be\label{equality for g in positive domain}
			\begin{split} \frac{m_2}{m_2-1}g^{m_2-1}-\mathcal{K}\ast f+\frac{2m_1}{M_2(d-2m_1)}\|g\|^{m_2}_{m_2}
				=0\,\,\,a.e. \,\,\,\text{in} \,\,\,\overline{\Omega}_{2+}.
			\end{split}
			\ee
			where we extend above equality to the whole space in the sense that
			\begin{equation*}\label{equality for w in whole case}
				\begin{split}
					\frac{m_2}{m_2-1}g^{m_2-1}=\left(\mathcal{K}\ast f-\frac{2m_1}{M_2(d-2m_1)}\|g\|^{m_2}_{m_2}\right)_{+}\,\,\,a.e. \,\,\,\text{in} \,\,\,\mathbb{R}^d.
				\end{split}
			\end{equation*}
			Since $g$ is radially symmetric and non-increasing, there exists $\rho\in(0,\infty]$ such that
			\begin{equation*}
				\begin{split}
					\Omega_{2+}\subset B(0,
					\rho)\,\,\,\text{and}\,\,\,\Omega_{20}\subset\mathbb{R}^d\backslash B(0,\rho),
				\end{split}
			\end{equation*}
			and from (\ref{equality for g in positive domain}) we obtain
			\begin{equation*}\label{equality for w in ball}
				\begin{split}
					\frac{m_2}{m_2-1}g^{m_2-1}=\mathcal{K}\ast f-\frac{2m_1}{M_2(d-2m_1)}\|g\|^{m_2}_{m_2}\,\,\,a.e. \,\,\,\text{in} \,\,\,B(0,\rho).
				\end{split}
			\end{equation*}
			Then such symmetric non-increasing minimizer $(f,g)\in S_{M_1}\times S_{M_2}$ of $\mathcal{F}$ satisfies the following Euler-Lagrange equalities
			\begin{equation}\label{Euler-Lagrange for f and g}
				\begin{cases}
						&\frac{m_1}{m_1-1}f^{m_1-1}(x)= \mathcal{K}\ast g(x)\,\,\,a.e. \,\,\,\text{in} \,\,\,\mathbb{R}^d,\,\,\,\\[3mm]
						&\frac{m_2}{m_2-1}g^{m_2-1}(x)=\mathcal{K}\ast f(x)-\frac{2m_1}{M_2(d-2m_1)}\|g\|^{m_2}_{m_2}\,\,\,a.e. \,\,\,\text{in} \,\,\,B(0,\rho).
				\end{cases}
			\end{equation}
			{\bf{Step 3.}}	\emph{The regularities of minimizer}. From $(\ref{Euler-Lagrange for f and g})_1$, one invokes the HLS inequality in Lemma {\ref{HLS inequality}} to see for $g\in L^1(\mathbb{R}^d)\cap L^{m_2}(\mathbb{R}^d)$  that
			\begin{equation*}
				\begin{split}{\label{increase of regularity of f}}
					f\in L^{p}(\mathbb{R}^d)\,\,\,\,\text{with}\,\,\,p\in \left[\frac{d(m_1-1)}{d-2},\frac{d(m_1-1)m_2}{d-2m_2}\right],
				\end{split}
			\end{equation*}
			where once more using the HLS inequality again, one concludes that
			\begin{equation*}\label{increase of regularity of Kf}
				\mathcal{K}\ast f\in L^{q}(\mathbb{R}^d)\,\,\,\text{with}\,\,\,q\in
				\begin{cases}\left[\frac{d(m_1-1)}{d-2m_1},
					\frac{d(m_1-1)m_2}{d-2m_1m_2}\right],\,\,\,&\text{if}\,\,\,d>2m_1m_2,\\[0.2cm]
					\left[\frac{d(m_1-1)}{d-2m_1},
					\infty\right),\,\,\,&\text{if}\,\,\,d\leq 2m_1m_2.\\[0.2cm]
				\end{cases}
			\end{equation*}
			In particular, $\mathcal{K}\ast f\in L^{\frac{m_2}{m_2-1}}(\mathbb{R}^d)$ since  $m_1+m_2=2m_1/d+m_1m_2\leq 2m_1m_2/d+m_1m_2$ and
			$$
			\frac{m_2}{m_2-1}\in\left[\frac{d(m_1-1)}{d-2m_1},\frac{d(m_1-1)m_2}{\left(d-2m_1m_2\right)_{+}}
			\right).
			$$
			Consequently, $g^{m_2-1}\in L^{\frac{m_2}{m_2-1}}(\mathbb{R}^d)$, which excludes $\rho=\infty$ in $(\ref{Euler-Lagrange for f and g})_2$. Hence $\rho<\infty$ and
			\begin{equation*} \label{equality in the ball2}
				\frac{m_2}{m_2-1}g^{m_2-1}(x)=\begin{cases}
					\mathcal{K}\ast f(x)-\frac{2m_1}{M_2(d-2m_1)}\|g\|^{m_2}_{m_2},\,\,\,&\text{if}\,\,\,|x|<\rho, \\[0.2cm]
					0,\,\,\,&\text{if}\,\,\,|x|>\rho
				\end{cases}
			\end{equation*}
			by the monotonicity of $g$. Moreover, a bootstrap argument ensures that
			\begin{equation*}
				\begin{split}
					(f,g)\in (L^{\infty}(\mathbb{R}^d))^2.
				\end{split}
			\end{equation*}
			Letting $\vartheta:=f^{m_1-1}$ and $\varsigma:=g^{m_2-1}$, we readily infer from $(\ref{Euler-Lagrange for f and g})_1$ that
			\begin{equation*} \label{Euler-Lagrange for f2}
				\vartheta(x)=\frac{m_1-1}{m_1} \mathcal{K}\ast\varsigma^{\frac{1}{m_2-1}}(x)\,\,\,a.e. \,\,\,\text{in} \,\,\,\mathbb{R}^d,\,\,\,
			\end{equation*}
			and invoke {\cite[Theorem 9.9]{GT83}} to have $\vartheta\in W^{2,r}(B(0,\rho))$ with $r\in (m_1,\infty)$ and $-\Delta \vartheta=\frac{m_1-1}{m_1}\varsigma^{\frac{1}{m_2-1}}$\,\,a.e.\,\,$x\in \mathbb{R}^d$. Furthermore, from the expression for $\varsigma$ such as
			
			\begin{equation*} \label{equality in the ball3}
				\varsigma(x)=
				 \frac{m_2-1}{m_2}\mathcal{K}\ast\vartheta^{\frac{1}{m_1-1}}(x)-\frac{2m_1(m_2-1)}{m_2M_2(d-2m_1)}\|\varsigma\|^{m_2/(m_2-1)}_{m_2/(m_2-1)},\,\,\,x\in B(0,\rho),
			\end{equation*}
			by means of the regularity of $\vartheta$ and {\cite[Lemma 4.2]{GT83}}, we obtain $\varsigma\in C^2(B(0,\rho))$ with $-\Delta\varsigma=\frac{m_2-1}{m_2}\vartheta^{\frac{1}{m_1-1}}$ in $B(0,\rho)$ and  {\cite[Lemma 4.1]{GT83}} ensures that $\varsigma\in C^1(\mathbb{R}^d)$. Then $\varsigma(x)=0$ if $|x|=\rho$ and $\varsigma$ is a classical solution to
			\be{\label{class solution for varsigma}}
			\begin{split}
				\begin{cases}
					&-\Delta\varsigma(x)=\frac{m_2-1}{m_2}\vartheta^{\frac{1}{m_1-1}}(x),\,\,\,\,x\in B(0,\rho),\\
					&\varsigma(x)=0,\,\,\,\,x\in \partial B(0,\rho).
				\end{cases}
			\end{split}
			\ee
			With the smoothness of $\varsigma$, {\cite[Lemma 4.2]{GT83}} applies so as to assert that  $\vartheta\in C^2(\mathbb{R}^d)$ and
			\be{\label{class solution for theta}}
			\begin{split}
				-\Delta\vartheta(x)=\frac{m_1-1}{m_1}\varsigma^{\frac{1}{m_2-1}}(x)\,
				,\,\,\,\,x\in \mathbb{R}^d.
			\end{split}
			\ee
			{\bf{Step 4.}}  \emph{Contradiction}. (\ref{class solution for varsigma})-(\ref{class solution for theta}) consist of the Lane-Emden system (\ref{Lane-Emden system}). However, it has been proved that there exists no non-trivial classical solution of (\ref{Lane-Emden system}) if ${\bf{m}}$ is on $L_1$, which makes a contradiction.

		\end{proof}
		
		\begin{remark}
			Let ${\bf{m}}$ be on $L_2$, there exists no non-zero minimizer for $\mathcal{F}$  in $S_{M_1}\times S_{M_2}$ with $M_1\leq M_{1c}$.
		\end{remark}

		\section{The global existence}
		
		This section deals with the global solvability of (\ref{TSTC}) in subcritical case. We first present a local existence and extensibility criterion of free energy solution to (\ref{TSTC}). Note that this theorem also provides simultaneous blow-up argument in Section 5.
		\begin{theorem}{\label{local existence theorem}}
			Let $m_1,m_2> 1$. Under assumption (\ref{intial data for u and w}) on the initial data $(u_0,w_0)$ with $\|u_0\|_1=M_1,\|w_0\|_1=M_2$, then there exists $T_{\max}\in(0,\infty]$  and  a free energy solution $(u,w)$ over $\mathbb{R}^d\times(0,T_{\max}) $ of (\ref{TSTC}) such that either $T_{\max}=\infty$ or $T_{\max}<\infty$ and
			\be\label{extension criterion1}
			\lim_{t\rightarrow T_{\max}}\left(\|u(\cdot,t)\|_{\infty}+\|w(\cdot,t)\|_{\infty}\right)=\infty.
			\ee
			Moreover, let ${\bf{m}}$ be subcritical or critical. Then if $T_{\max}<\infty$,
			\be\label{extension criterion2}
			\lim_{t\rightarrow T_{\max}}\|u(\cdot,t)\|_{m_1}=\lim_{t\rightarrow T_{\max}}\|w(\cdot,t)\|_{m_2}=\infty.
			\ee
		\end{theorem}
		\begin{proof}
			For $(u_0,w_0)$ satisfying (\ref{intial data for u and w}), local existence and (\ref{extension criterion1}) can be proved by approximation arguments (similar to those in the proof of Theorem 1.1 in \cite{Sugiyama06-DIE} for instance). To see (\ref{extension criterion2}), since the solution is globally solved if both $\|u\|_{m_1}$ and $\|w\|_{m_2}$ are uniform bound in subcritical or critical case due to Lemmas \ref{the condition for global existence}-\ref{proof of FES}, then it is sufficient to show that the two terms $\|u\|_{m_1}$ and $\|w\|_{m_2}$ are governed  by each other with some constants. 	
			
			Since
			\be\label{inequality for free energy}
			\begin{split}
				\frac{1}{m_1-1}\int_{\mathbb{R}^d}u^{m_1}+\frac{1}{m_2-1}\int_{\mathbb{R}^d}w^{m_2}\leq c_d\mathcal{H}[u,w]+\mathcal{F}[u_0,w_0],
			\end{split}
			\ee
			then it needs to control the term $\mathcal{H}$ at the right side of (\ref{inequality for free energy}). For $m\in\left(1,d/2\right)$ satisfying (\ref{connection between m_1 and m_21}),  Lemma \ref{estimate for H lem} yields that
			\begin{equation}\label{inequality for H}
				\begin{split}
					 |\mathcal{H}[f,g]|&\leq\eta\|f\|^{m}_{{m}}+C\eta^{-\frac{1}{m-1}}\|g\|^{\frac{mm_2+2mm_2/d-m-m_2}{(m-1)(m_2-1)}}_1\|g\|^{\frac{m_2-2mm_2/d}{(m-1)(m_2-1)}}_{m_2}
				\end{split}
			\end{equation}
			for some $f\in L^{m}(\mathbb{R}^d)$ and $g\in L^{1}(\mathbb{R}^d)\cap L^{m_2}(\mathbb{R}^d)$ with $\eta>0$.
			If $m_1<d/2$, choosing $m=m_1$ in (\ref{inequality for H}), then
			\begin{equation*}
				\begin{split}
					\frac{1}{m_1-1}&\int_{\mathbb{R}^d}u^{m_1}+\frac{1}{m_2-1}\int_{\mathbb{R}^d}w^{m_2}\\
					\leq& c_d\eta\|u\|^{m_1}_{{m_1}}
					 +c_dC\eta^{-\frac{1}{m_1-1}}M^{\frac{m_1m_2+2m_1m_2/d-m_1-m_2}{(m_1-1)(m_2-1)}}_2\|w\|^{\frac{m_2-2m_1m_2/d}{(m_1-1)(m_2-1)}}_{m_2}
					+\mathcal{F}[u_0,w_0]\\
					\leq& c_d\eta\|u\|^{m_1}_{{m_1}}
					+c_dC\eta^{-\frac{1}{m_1-1}}\|w\|^{m_2}_{m_2}+C
				\end{split}
			\end{equation*}
			by Young's inequality, since
			\begin{equation*}
				\begin{split}
					\frac{m_2-2m_1m_2/d}{(m_1-1)(m_2-1)}\leq m_2
				\end{split}
			\end{equation*}	
			if $m_1m_2+2m_1/d\geq m_1+m_2$ holds. Taking $\eta$ small enough, we have
			\begin{equation}{\label{u controlled by w}}
				\begin{split}
					\|u(t)\|^{m_1}_{m_1}\leq C\|w(t)\|^{m_2}_{m_2}+C\,\,\,\,\text{for}\,\,\,\,t\in(0,T_{\max})
				\end{split}
			\end{equation}
			and  if $\eta$ is sufficiently large, we see that
			\begin{equation}{\label{w controlled by u}}
				\begin{split}
					\|w(t)\|^{m_2}_{m_2}\leq C'\|u(t)\|^{m_1}_{m_1}+C'\,\,\,\,\text{for}\,\,\,\,t\in(0,T_{\max}).
				\end{split}
			\end{equation}
			Therefore,  (\ref{extension criterion2}) holds by (\ref{extension criterion1}), (\ref{u controlled by w})-(\ref{w controlled by u}).

			However, if $m_1\geq d/2$,  we pick $m\in\left(1,d/2\right)$ such that
			\begin{equation*}
				\begin{split}
					\frac{m_2}{m_2+2/d-1}< m<d/2,	
				\end{split}
			\end{equation*}
			and next take interpolation inequality to find that
			\begin{equation*}
				\begin{split}
					\|u\|^m_{m}\leq& \|u\|^{\frac{m_1-m}{m_1-1}}_1\|u\|^{\frac{m_1(m-1)}{m_1-1}}_{m_1}.
				\end{split}
			\end{equation*}
			Upon
			\begin{equation*}
				\begin{split}
					\frac{m_2-2mm_2/d}{(m-1)(m_2-1)}<m_2,	
				\end{split}
			\end{equation*}
			then (\ref{inequality for H}) implies that
			\begin{equation}\label{estimate for H2}
				\begin{split}
					|\mathcal{H}[u,w]|
					&\leq\eta \|u\|^{\frac{m_1-m}{m_1-1}}_1\|u\|^{\frac{m_1(m-1)}{m_1-1}}_{m_1}+C\eta^{-\frac{1}{m-1}}
					\|w\|^{\frac{mm_2+2mm_2/d-m-m_2}{(m-1)(m_2-1)}}_1\|w\|^{\frac{m_2-2mm_2/d}{(m-1)(m_2-1)}}_{m_2}\\
					&=\eta M^{\frac{m_1-m}{m_1-1}}_1\|u\|^{\frac{m_1(m-1)}{m_1-1}}_{m_1}+C\eta^{-\frac{1}{m-1}}M^{\frac{mm_2+2mm_2/d-m-m_2}{(m-1)(m_2-1)}}_2\|w\|^{\frac{m_2-2mm_2/d}{(m-1)(m_2-1)}}_{m_2}\\
					&\leq \eta \|u\|^{m_1}_{m_1}+\eta^{-\frac{1}{m-1}}\|w\|^{m_2}_{m_2}+C\\
				\end{split}
			\end{equation}
			with $\|u\|_1=M_1$ and $\|w\|_1=M_2$. Hence  (\ref{u controlled by w})-(\ref{w controlled by u}) are valid by picking suitable $\eta>0$. By the same token, the  case $m_1m_2+2m_2/d\geq m_1+m_2$ is also true for both $m_2<d/2$ and $m_2\geq d/2$. The proof is finished.
			
		\end{proof}
		
		The global existence result in subcritical case is the subject of our next theorem.
		\begin{theorem}{\label{global existence of more general case}}
			Let  $m_1,m_2>1$. Suppose that the initial data $(u_0,w_0)$ with $\|u_0\|_1=M_1,\|w_0\|_1=M_2$ fulfills (\ref{intial data for u and w}). Then if
			${\bf{m}}$ is subcritical, (\ref{TSTC}) has a global free energy solution given in Definition \ref{free energy solution}. 	
		\end{theorem}
		
		\begin{remark}
			If $m_1\geq d/2$ or $m_2\geq d/2$, the conclusion in Theorem  \ref{global existence of more general case} holds for all $m_2>1$ or $m_1>1$.
		\end{remark}
		\begin{proof}
			In the  case
			$m_1m_2+2m_1/d>m_1+m_2$ and $m_1<d/2$,  since
			$\frac{m_2-2m_1m_2/d}{(m_1-1)(m_2-1)}<m_2$, then Lemma {\ref{estimate for H lem}} warrants that
			\begin{equation*}
				\begin{split}
					|\mathcal{H}[u,w]|
					 &\leq\frac{1}{2c_d(m_1-1)}\|u\|^{m_1}_{{m_1}}+C\|w\|^{\frac{m_1m_2+2m_1m_2/d-m_1-m_2}{(m_1-1)(m_2-1)}}_1\|w\|^{\frac{m_2-2m_1m_2/d}{(m_1-1)(m_2-1)}}_{m_2}\\
					&\leq \frac{1}{2c_d(m_1-1)}\|u\|^{m_1}_{{m_1}}+\frac{1}{2c_d(m_2-1)}\|w\|^{m_2}_{{m_2}}+C
				\end{split}
			\end{equation*}	
			by Young's inequality. Then substituting (\ref{inequality for free energy}) into above, we have
			\begin{equation*}
				\begin{split}
					\frac{1}{m_1-1}\int_ {\mathbb{R}^d}u^{m_1}dx+&\frac{1}{m_2-1}\int_ {\mathbb{R}^d}w^{m_2}dx\\
					\leq&\frac{1}{2(m_1-1)}\int_ {\mathbb{R}^d}u^{m_1}dx+\frac{1}{2(m_2-1)}\int_ {\mathbb{R}^d}w^{m_2}dx+C.
				\end{split}
			\end{equation*}
			As a corollary,
			\be{\label{boundedness for um1 and wm2}}
			\begin{split}
				\|u\|_{m_1}\leq C\,\,\,\,\text{and}\,\,\,\,\|w\|_{m_2}\leq C.
			\end{split}
			\ee
			If $m_1\geq \frac{d}{2}$, we
			recalculate (\ref{estimate for H2}) carefully and also have (\ref{boundedness for um1 and wm2}), in which the global existence of free energy solution is immediate from Theorem \ref{local existence theorem}.  The other case $m_1m_2+2m_2/d> m_1+m_2$ is similar.
		\end{proof}

		Also on the critical lines, we obtain  global existence results reading as
		\begin{theorem}{\label{global existence for subcritical case}}
			Let ${\bf{m}}$ be on $L_1$, and let $(u,w)$ be a free energy solution of (\ref{TSTC}) with $(u_0,w_0)$ satisfying  (\ref{intial data for u and w}) on $[0,T_{\max})$ with $T_{\max}$ given in Theorem {\ref{local existence theorem}}. If
			\be\label{subcritical condition}
			\begin{split}
				M_2<M_{2c},
			\end{split}
			\ee
			then $T_{\max}=\infty$.
			The subcritical condition (\ref{subcritical condition}) will be replaced by $M_1<M_{1c}$ on $L_2$.
			Moreover, if ${\bf{m}}$ is ${\bf{I}}$, one has $T_{\max}=\infty$ if $M_1M_2<M^2_c$.
		\end{theorem}
		\begin{proof}
			We just infer from (\ref{inequality for free energy2}) and Lemma \ref{infimum for FEF4} that
			\begin{equation*}
				\begin{split}
					 \left(c_dC_{*}\right)^{\frac{m_1}{m_1-1}}&(m_1-1)^{\frac{m_1}{m_1-1}}m^{-\frac{m_1}{m_1-1}}_1\left(M^{\frac{2m_1}{d(m_1-1)}}_{2c}-M^{\frac{2m_1}{d(m_1-1)}}_{2}\right)\|w\|^{m_2}_{m_2}\\
					&\leq\mathcal{F}[u,w]\leq\mathcal{F}[u_0,w_0].
				\end{split}
			\end{equation*}
			Due to (\ref{subcritical condition}), there exists $C>0$ such that for all $t\in[0,T_{\max})$ we have $\|w\|_{m_2}\leq C$ . Then the extensibility criterion in Theorem {\ref{local existence theorem}} makes sure that $T_{\max}=\infty$.
			The other cases can be similarly obtained.
		\end{proof}

		\section{Blow up}
		Our last section concerns finite-time blow-up phenomenon when ${\bf{m}}$ is critical or super-critical. These results actually show that lines $L_i$, $i=1,2$ are optimal in view of the global existence for sub-critical case.  The following second moment of solutions can be achieved in a straightforward computation.
		\begin{lemma}{\label{second moment}}
			Let $(u_0,w_0)$ satisfy (\ref{intial data for u and w}), and let $(u,w)$ be a free energy solution of (\ref{TSTC}) on $[0,T_{\max})$ with $T_{\max}\in(0,\infty]$. Then
			$$
			\frac{d}{dt}I(t)=G(t)\,\,\,\text{for all}
			\,\,\,\,t\in(0,T_{\max}),
			$$
			where
			$$
			I(t):=\int_{\mathbb{R}^d}|x|^2\left(u(x,t)+w(x,t)\right)dx
			$$
			and
			\begin{align*}
				G(t):=&2d\int_{\mathbb{R}^d}u^{m_1}(x,t)dx
				+2d\int_{\mathbb{R}^d}w^{m_2}(x,t)dx\\
				&-2c_d(d-2)\iint_{\mathbb{R}^d\times\mathbb{R}^d}\frac{u(x,t)w(y,t)}{|x-y|^{d-2}}dxdy.
			\end{align*}	
		\end{lemma}
		\begin{proof}
			We differentiate the second moment to see that
			\begin{equation*}
				\begin{split}
					\frac{d}{dt}&\int_{\mathbb{R}^d}|x|^2(u(x,t)+w(x,t))dx\\
					=&\int_{\mathbb{R}^d}|x|^2(\Delta u^{m_1}-\nabla\cdot(u\nabla v))dx
					+\int_{\mathbb{R}^d}|x|^2(\Delta w^{m_2}-\nabla\cdot(w\nabla z))dx\\
					=&2d\int_{\mathbb{R}^d}u^{m_1}(x,t)dx+2d\int_{\mathbb{R}^d}w^{m_2}(x,t)dx\\
					&+2\iint_{\mathbb{R}^d\times\mathbb{R}^d}[x\cdot\nabla \mathcal{K}(x-y)]u(x,t)w(y,t)dxdy\\
					&+2\iint_{\mathbb{R}^d\times\mathbb{R}^d}[x\cdot\nabla \mathcal{K}(x-y)]u(y,t)w(x,t)dxdy.
				\end{split}
			\end{equation*}
			With $\mathcal{K}(x)=c_d\frac{1}{|x|^{d-2}}$, we have
			\begin{equation*}
				\begin{split}		
					2\iint_{\mathbb{R}^d\times\mathbb{R}^d}&[x\cdot\nabla \mathcal{K}(x-y)]u(x,t)w(y,t)dxdy\\
					=&-2c_d(d-2)\iint_{\mathbb{R}^d\times\mathbb{R}^d}\frac{(x-y)\cdot x}{|x-y|^d}u(x,t)w(y,t)dxdy\\
					=&-2c_d(d-2)\iint_{\mathbb{R}^d\times\mathbb{R}^d}\frac{|x|^2}{|x-y|^d}u(x,t)w(y,t)dxdy\\
					&+2c_d(d-2)\iint_{\mathbb{R}^d\times\mathbb{R}^d}\frac{x\cdot y}{|x-y|^d}u(x,t)w(y,t)dxdy\\
					=&-c_d(d-2)\iint_{\mathbb{R}^d\times\mathbb{R}^d}\frac{|x|^2}{|x-y|^d}u(x,t)w(y,t)dxdy\\
					&-c_d(d-2)\iint_{\mathbb{R}^d\times\mathbb{R}^d}\frac{|y|^2}{|x-y|^d}u(y,t)w(x,t)dxdy\\
					&+2c_d(d-2)\iint_{\mathbb{R}^d\times\mathbb{R}^d}\frac{x\cdot y}{|x-y|^d}u(x,t)w(y,t)dxdy\\
				\end{split}
			\end{equation*}
			and
			\begin{equation*}
				\begin{split}		
					2\iint_{\mathbb{R}^d\times\mathbb{R}^d}&[x\cdot\nabla \mathcal{K}(x-y)]u(y,t)w(x,t)dxdy\\
					=&-c_d(d-2)\iint_{\mathbb{R}^d\times\mathbb{R}^d}\frac{|x|^2}{|x-y|^d}u(y,t)w(x,t)dxdy\\
					&-c_d(d-2)\iint_{\mathbb{R}^d\times\mathbb{R}^d}\frac{|y|^2}{|x-y|^d}u(x,t)w(y,t)dxdy\\
					&+2c_d(d-2)\iint_{\mathbb{R}^d\times\mathbb{R}^d}\frac{x\cdot y}{|x-y|^d}u(x,t)w(y,t)dxdy.
				\end{split}
			\end{equation*}
			Combining above equations, it follows that
			\begin{equation*}
				\begin{split}		
					 \frac{d}{dt}\int_{\mathbb{R}^d}|x|^2(u(x,t)+&w(x,t))dx=2d\int_{\mathbb{R}^d}u^{m_1}(x,t)dx+2d\int_{\mathbb{R}^d}w^{m_2}(x,t)dx\\
					&-c_d(d-2)\iint_{\mathbb{R}^d\times\mathbb{R}^d}\frac{|x|^2+|y|^2}{|x-y|^d}u(x,t)w(y,t)dxdy\\
					&-c_d(d-2)\iint_{\mathbb{R}^d\times\mathbb{R}^d}\frac{|x|^2+|y|^2}{|x-y|^d}u(y,t)w(x,t)dxdy\\
					&+4c_d(d-2)\iint_{\mathbb{R}^d\times\mathbb{R}^d}\frac{x\cdot y}{|x-y|^d}u(x,t)w(y,t)dxdy\\
					=&2d\int_{\mathbb{R}^d}u^{m_1}(x,t)dx+2d\int_{\mathbb{R}^d}w^{m_2}(x,t)dx\\
					&-2c_d(d-2)\iint_{\mathbb{R}^d\times\mathbb{R}^d}\frac{u(x,t)w(y,t)}{|x-y|^{d-2}}dxdy,
				\end{split}
			\end{equation*}
			which readily implies the lemma.
		\end{proof}

		We construct initial data which ensures the nonnegativity of $G(0)$.
		\begin{lemma}{\label{initial data ensruing blowup}}
			Let ${\bf{m}}$ be critical or super-critical.	There exists initial data $(u_0,w_0)$ satisfying (\ref{intial data for u and w}), and fulfilling
			\be\label{blow-up condition for initial data2}
			\begin{split}
				 &\frac{\left(\int_{\mathbb{R}^d}u^{\frac{(m_1+m_2-m_1m_2)d}{2m_2}}_0dx\right)^{\frac{2m_2}{(m_1+m_2-m_1m_2)d}}\left(\int_{\mathbb{R}^d}w^{\frac{(m_1+m_2-m_1m_2)d}{2m_1}}_0dx\right)^{\frac{2m_1}{(m_1+m_2-m_1m_2)d}}}{\left(\int_{\mathbb{R}^d}u^{\frac{(m_1+m_2-m_1m_2)d}{2m_2}}_0dx\right)^{\frac{2m_1m_2}{(m_1+m_2-m_1m_2)d}}+\left(\int_{\mathbb{R}^d}w^{\frac{(m_1+m_2-m_1m_2)d}{2m_1}}_0dx\right)^{\frac{2m_1m_2}{(m_1+m_2-m_1m_2)d}}}\\
				&>\begin{cases}
					N_0,\,\,\,&\text{if}\,\,\,m_1m_2+2\max\{m_1,m_2\}/d\leq m_1+m_2<m_1m_2+2m_1m_2/d,\\[0.2cm]
					2N_0,\,\,\,&\text{if}\,\,\, m_1+m_2\geq m_1m_2+2m_1m_2/d,
				\end{cases}
			\end{split}
			\ee
		and
			\be\label{negativity of G}
			\begin{split}
				G(0)<0,
			\end{split}
			\ee	
			where
			\begin{align*}
				 N_0=&\frac{\left(d/c_d\right)^{2-2/d}}{2^{1+2/d}(d-2)}\left(1+\frac{2m_1}{(m_1+m_2-m_1m_2)d}\right)\left(1+\frac{2m_2}{(m_1+m_2-m_1m_2)d}\right)
			\end{align*}
			and  $G$ is given in Lemma {\ref{second moment}}.
		\end{lemma}
		\begin{proof}
			Consider the following functions having the same compact support as initial data of form
			\be\label{construction for initial data}
			\begin{split}
				u_0(x)=&A\left(1-\frac{|x|^d}{a^d}\right)^{\iota_1}_{+},\,\,\,x\in\mathbb{R}^d,\\
				w_0(x)=&B\left(1-\frac{|x|^d}{a^d}\right)^{\iota_2}_{+},\,\,\,x\in\mathbb{R}^d,
			\end{split}
			\ee
			with
			\be\label{definition of iota1 and iota2}
			\begin{split}
				\iota_1:=\frac{2m_2}{(m_1+m_2-m_1m_2)d}\,\,\,\,\text{and}\,\,\, \iota_2:=\frac{2m_1}{(m_1+m_2-m_1m_2)d},
			\end{split}
			\ee
			where $A,B>0$ denote the maximum of the supports and $a>0$ denotes the size of the supports of initial data. Such constructions in (\ref{construction for initial data}) are inspired by {\cite[Section 6]{Sugiyama07-ADE}} which deals with one-single population Keller-Segel system.
			
			In the {\bf{Case 1}}: $m_1m_2+2\max\{m_1,m_2\}/d\leq m_1+m_2<m_1m_2+2m_1m_2/d,$ one has
			\be\label{upper for u_0 1}
			\begin{split}
				\int_{\mathbb{R}^d}u^{m_1}_0dx&=A^{m_1}	 \int_{\mathbb{R}^d}\left(1-\frac{|x|^d}{a^d}\right)^{\frac{2m_1m_2}{(m_1+m_2-m_1m_2)d}}_{+}dx\\
				&=A^{m_1}	 \int_{\mathbb{R}^d}\left(1-\frac{|x|^d}{a^d}\right)^{\frac{2m_1m_2}{(m_1+m_2-m_1m_2)d}-1}_{+}\left(1-\frac{|x|^d}{a^d}\right)_{+}dx\\
				&\leq A^{m_1}	\int_{\mathbb{R}^d}\left(1-\frac{|x|^d}{a^d}\right)_{+}dx\\
				&=c_da^dA^{m_1}/(2d)
			\end{split}
			\ee
			and
			$$
			\int_{\mathbb{R}^d}w^{m_2}_0dx\leq c_da^dB^{m_2}/(2d).
			$$
			For the {\bf{Case 2}}: $m_1+m_2>m_1m_2+2m_1m_2/d$,
			\begin{equation*}
				\begin{split}
					\int_{\mathbb{R}^d}u^{m_1}_0dx	 	
					&\leq A^{m_1}	\int_{|x|<a}1dx=c_da^dA^{m_1}/d,\\
					\int_{\mathbb{R}^d}w^{m_2}_0dx	 	
					&\leq c_da^dB^{m_2}/d.
				\end{split}
			\end{equation*}
			The coupled term can be estimated as
			\be\label{lower for u_0w_0}
			\begin{split}
				\iint_{\mathbb{R}^d\times\mathbb{R}^d}\frac{u_0(x)w_0(y)}{|x-y|^{d-2}}dxdy\geq& \min_{|x|,|y|\leq a}|x-y|^{-(d-2)}
				\int_{\mathbb{R}^d}u_0(x)dx\cdot\int_{\mathbb{R}^d}w_0(x)dx\\
				\geq& a^{-(d-2)} \int_{\mathbb{R}^d}A\left(1-\frac{|x|^d}{a^d}\right)^{\iota_1}_{+}dx
				\cdot \int_{\mathbb{R}^d}B\left(1-\frac{|x|^d}{a^d}\right)^{\iota_2}_{+}dx\\
				=& \frac{c^2_da^{d+2}}{d^2(1+\iota_1)(1+\iota_2)}
				AB .
			\end{split}
			\ee
			Since
			\be\label{upper estimate for G}
			\begin{split}
				G(0)\leq c_da^dA^{m_1}+c_da^dB^{m_2}-\frac{2c^3_da^{d+2}(d-2)}{d^2(1+\iota_1)(1+\iota_2)}AB
			\end{split}
			\ee
			by (\ref{upper for u_0 1})-(\ref{lower for u_0w_0}),
			to show (\ref{negativity of G}), it only needs to show the right side of  (\ref{upper estimate for G}) is negative such that
			\be\label{inequality for AB}
			\begin{split}
				\frac{AB}{A^{m_1}+B^{m_2}}a^2> N_1
			\end{split}
			\ee
			with
			\begin{equation*}
				\begin{split}
					N_1=&\frac{d^2\left(1+\iota_1\right)\left(1+\iota_2\right)}{2c^2_d(d-2)}
				\end{split}
			\end{equation*}
			in the {\bf{Case 1}}, whereas the right side will be replaced by $2N_1$ in the {\bf{Case 2}}.
			
			Since
			\begin{equation*}
				\begin{split}
					 \int_{\mathbb{R}^d}u^{\frac{1}{\iota_1}}_0dx=A^{\frac{1}{\iota_1}}\int_{\mathbb{R}^d}\left(1-\frac{|x|^d}{a^d}\right)_{+}dx&=c_da^dA^{\frac{1}{\iota_1}}/(2d),\\
					 \int_{\mathbb{R}^d}w^{\frac{1}{\iota_2}}_0dx=B^{\frac{1}{\iota_2}}\int_{\mathbb{R}^d}\left(1-\frac{|x|^d}{a^d}\right)_{+}dx&=c_da^dB^{\frac{1}{\iota_2}}/(2d)
				\end{split}
			\end{equation*}
			implies that
			\begin{equation*}
				\begin{split}
					 A=&\left(\frac{2d}{c_d}\int_{\mathbb{R}^d}u^{\frac{1}{\iota_1}}_0dx\right)^{\iota_1}a^{-\iota_1d},\,\,\,\,
					B=\left(\frac{2d}{c_d}\int_{\mathbb{R}^d}w^{\frac{1}{\iota_2}}_0dx\right)^{\iota_2}a^{-\iota_2d},
				\end{split}
			\end{equation*}
			then (\ref{inequality for AB}) can be rewritten as
			\begin{align*}
				\frac{AB}{A^{m_1}+B^{m_2}}a^2
				 &=\frac{\left(\frac{2d}{c_d}\int_{\mathbb{R}^d}u^{\frac{1}{\iota_1}}_0dx\right)^{\iota_1}\left(\frac{2d}{c_d}\int_{\mathbb{R}^d}w^{\frac{1}{\iota_2}}_0dx\right)^{\iota_2}}{{\left(\frac{2d}{c_d}\int_{\mathbb{R}^d}u^{\frac{1}{\iota_1}}_0dx\right)^{\iota_1m_1}}+\left(\frac{2d}{c_d}\int_{\mathbb{R}^d}w^{\frac{1}{\iota_2}}_0dx\right)^{\iota_2m_2}}\\
				 &=\left(\frac{2d}{c_d}\right)^{\frac{2}{d}}\frac{\left(\int_{\mathbb{R}^d}u^{\frac{1}{\iota_1}}_0dx\right)^{\iota_1}\left(\int_{\mathbb{R}^d}w^{\frac{1}{\iota_2}}_0dx\right)^{\iota_2}}{{\left(\int_{\mathbb{R}^d}u^{\frac{1}{\iota_1}}_0dx\right)^{\iota_1m_1}}+\left(\int_{\mathbb{R}^d}w^{\frac{1}{\iota_2}}_0dx\right)^{\iota_2m_2}}\\
				&> N_1\,\,\,\,
				\left(\text{or}\,\,\,2N_1\,\,\, \text{for the {\bf{Case 2}}}\right).
			\end{align*}
			Therefore, we have
			\begin{equation*}\label{blow-up condition for initial data}
				\begin{split}
					 &\frac{\left(\int_{\mathbb{R}^d}u^{\frac{1}{\iota_1}}_0dx\right)^{\iota_1}\left(\int_{\mathbb{R}^d}w^{\frac{1}{\iota_2}}_0dx\right)^{\iota_2}}{{\left(\int_{\mathbb{R}^d}u^{\frac{1}{\iota_1}}_0dx\right)^{\iota_1m_1}}+\left(\int_{\mathbb{R}^d}w^{\frac{1}{\iota_2}}_0dx\right)^{\iota_2m_2}}\\
					&>\begin{cases}
						N_2,\,\,\,&\text{if}\,\,\,m_1m_2+\frac{2}{d}\max\{m_1,m_2\}\leq m_1+m_2<m_1m_2+\frac{2}{d}m_1m_2,\\[0.2cm]
						2N_2,\,\,\,&\text{if}\,\,\, m_1+m_2\geq m_1m_2+\frac{2}{d}m_1m_2,
					\end{cases}
				\end{split}
			\end{equation*}
			with
			\begin{equation*}\label{}
				\begin{split}
					N_2=&\frac{\left(d/c_d\right)^{2-2/d}}{2^{1+2/d}(d-2)} \left(1+\frac{2m_1}{(m_1+m_2-m_1m_2)d}\right)
					\left(1+\frac{2m_2}{(m_1+m_2-m_1m_2)d}\right),
				\end{split}
			\end{equation*}
			which yields $G(0)<0$ with $N_0=N_2$.
		\end{proof}
		
		The blow-up results state that
		\begin{theorem}{\label{results for blowup down the Line 1 and Line 2}}
			Let ${\bf{m}}$ be critical or super-critical. Then one can find some  initial data  $(u_0,w_0)$ satisfying (\ref{intial data for u and w}) such that free energy solution $(u,w)$ of (\ref{TSTC}) with $(u,w)\mid_{t=0}=(u_0,w_0)$
			blows up in finite time.
		\end{theorem}
		\begin{proof}
			For a given initial data $(u_0,w_0)$ in (\ref{construction for initial data}) satisfying (\ref{blow-up condition for initial data2}), then $G(0)<0$ from Lemma \ref{initial data ensruing blowup}. By the continuity argument, there exists $T^*>0$ such that
			\begin{equation*}\label{}
				\begin{split}
					G(t)<G(0)/2\,\,\,\text{for all}\,\,\,t\in[0,T^*],
				\end{split}
			\end{equation*}\label{}
			where from Lemma \ref{second moment}, one obtains $\frac{d}{dt}I(t)<G(0)/2$ for all  $t\in[0,T^*]$.
			Integrating by parts, it follows that
			\be\label{comparation between I0 and G0}
			\begin{split}
				I(T^*)<I(0)+G(0)T^*/2.
			\end{split}
			\ee
			As
			\be\label{upper estimate for I0}
			\begin{split}
				 I(0)=&\int_{\mathbb{R}^d}|x|^2\left(A\left(1-\frac{|x|^d}{a^d}\right)^{\iota_1}_{+}+B\left(1-\frac{|x|^d}{a^d}\right)^{\iota_2}_{+}\right)dx\\
				=&A\int_{|x|\leq a}|x|^2\left(1-\frac{|x|^d}{a^d}\right)^{\iota_1}dx+B\int_{|x|\leq a}|x|^2\left(1-\frac{|x|^d}{a^d}\right)^{\iota_2}dx\\
				 =&c_dA\int^a_0\left(1-\frac{r^d}{a^d}\right)^{\iota_1}r^{d+1}dr+c_dB\int^a_0\left(1-\frac{r^d}{a^d}\right)^{\iota_2}r^{d+1}dr\\
				 =&(c_da^{d+2}A)/d\int^1_0\left(1-r\right)^{\iota_1}r^{2/d}dr+(c_da^{d+2}B)/d\int^1_0\left(1-r\right)^{\iota_2}r^{2/d}dr\\
				=&(c_da^{d+2}AN_3)/d+(c_da^{d+2}BN_4)/d\\
			\end{split}
			\ee
			with $\iota_1,\iota_2$ given in (\ref{definition of iota1 and iota2}) and
			\begin{equation*}
				\begin{split}
					 N_3:=\int^1_0\left(1-r\right)^{\iota_1}r^{2/d}dr<\infty\,\,\,\text{and}\,\,\,\,N_4:=\int^1_0\left(1-r\right)^{\iota_2}r^{2/d}dr<
					\infty,
				\end{split}
			\end{equation*}
			then inserting (\ref{upper estimate for G}) and (\ref{upper estimate for I0}) into (\ref{comparation between I0 and G0}),  the right side of (\ref{comparation between I0 and G0}) should be negative if we may fix small $a>0$ such that
			\begin{equation*}
				\begin{split}
					\frac{T^*}{2}&\cdot\left[ \frac{2c^3_da^{d+2}(d-2)}{d^2(1+\iota_1)(1+\iota_2)}AB-c_da^dA^{m_1}-c_da^dB^{m_2}\right]\\
					&\geq(c_da^{d+2}AN_3)/d+(c_da^{d+2}BN_4)/d.
				\end{split}
			\end{equation*}
			More precisely, if
			\begin{align*}
				&\frac{dT^*}{2}\cdot\Big[\frac{2^{1+2/d}(d-2)}{(1+\iota_1)(1+\iota_2)}\left(\frac{c_d}{d}\right)^{2-2/d}
				 \cdot\left(\int_{\mathbb{R}^d}u^{\frac{1}{\iota_1}}_0dx\right)^{\iota_1}\left(\int_{\mathbb{R}^d}w^{\frac{1}{\iota_2}}_0dx\right)^{\iota_2}\\ &-\left(\int_{\mathbb{R}^d}u^{\frac{1}{\iota_1}}_0dx\right)^{m_1\iota_1}-\left(\int_{\mathbb{R}^d}w^{\frac{1}{\iota_2}}_0dx\right)^{m_2\iota_2}\Big]\\
				&\geq \left(\frac{2d}{c_d}\right)^{(1-m_1)\iota_1}\left(\int_{\mathbb{R}^d}u^{\frac{1}{\iota_1}}_0dx\right)^{\iota_1}a^{d\iota_2}N_3+\left(\frac{2d}{c_d}\right)^{(1-m_2)\iota_2}\left(\int_{\mathbb{R}^d}w^{\frac{1}{\iota_2}}_0dx\right)^{\iota_2}a^{d\iota_1}N_4,
			\end{align*}
			this leads to a contradiction after time $T^*$ since $I(t)$ is always nonnegative for all $t>0$. Hence the solutions blow up in finite time.
		\end{proof}

		If ${\bf{m}}$ is ${\bf{I}}$, Theorem {\ref{results for blowup down the Line 1 and Line 2} shows that the blow up condition (\ref{blow-up condition for initial data2}) can be written as
			\be\label{blow up condition for I}
			\begin{split}
				&\frac{M_1M_2}{M^{m_c}_1+M^{m_c}_2}> \frac{1}{2(d-2)}\cdot\left(\frac{2d}{c_d}\right)^{m_c},
			\end{split}
			\ee
			since
			\begin{equation*}
				\begin{split}
					\frac{d}{dt}I(t)=G(t)=2(d-2)\mathcal{F}[u(t),w(t)]\leq 2(d-2)\mathcal{F}[u_0,w_0]=G(0)<0
				\end{split}
			\end{equation*}
			if (\ref{blow up condition for I}) holds, then the second moment will be negative after some time and it contradicts the non-negativity of $u$ and $w$.
			
			We improve blow-up arguments if ${\bf{m}}$ is ${\bf{I}}$ by using a different method and summarize the blows up results on the lines $L_1,L_2$ and intersection point $\bf{I}$ as
			\begin{theorem}{\label{blow-up theorem}}
			Let ${\bf{m}}$ be critical.	Suppose that $(u,w)$ is a free energy solution of (\ref{TSTC}) with $\|u_0\|_1=M_1$, $\|w_0\|_1=M_2$ fulfilling (\ref{intial data for u and w}).
				
				If ${\bf{m}}$ is on $L_1$, for sufficiently small size of the supports of $(u_0,w_0)$ one asserts that blow up happens if
				$$
				 \frac{\left(\int_{\mathbb{R}^d}u^{m_1/m_2}_0dx\right)^{m_2/m_1}\left(\int_{\mathbb{R}^d}w_0dx\right)}{\left(\int_{\mathbb{R}^d}u^{m_1/m_2}_0dx\right)^{m_2}+\left(\int_{\mathbb{R}^d}w_0dx\right)^{m_2}}>N_0
				$$
				with $N_0$ given in Lemma \ref{initial data ensruing blowup}.
				
				If ${\bf{m}}$ is on $L_2$, for sufficiently small size of the supports of $(u_0,w_0)$  blow-up solution can be constructed if
				$$
				 \frac{\left(\int_{\mathbb{R}^d}u_0dx\right)\left(\int_{\mathbb{R}^d}w^{m_2/m_1}_0dx\right)^{m_1/m_2}}{\left(\int_{\mathbb{R}^d}u_0dx\right)^{m_1}+\left(\int_{\mathbb{R}^d}w^{m_2/m_1}_0dx\right)^{m_1}}>N_0.
				$$
				If ${\bf{m}}$ is $\bf{I}$, blow up occurs if
				$$
				M_1M_2/(M^{m_c}_1+M^{m_c}_2)>M^{2/d}_c/2.
				$$
				Finally, let $(u,w)$ blow up in finite time  $T_{\max}$.  Then $T_{\max}<\infty$ implies that
				$$
				\lim_{t\rightarrow T_{\max}}\|u\|_{m_1}=\lim_{t\rightarrow T_{\max}}\|w\|_{m_2}=\infty.
				$$
			\end{theorem}
			\begin{proof}
				The asserted blow-up conditions on the lines $L_1$ and $L_2$ just follow from Lemma {\ref{initial data ensruing blowup}} and Theorem \ref{results for blowup down the Line 1 and Line 2}. If $\bf{m}$ is $\bf{I}$, note that for any $M^*_1>0$ and $M^*_2>0$ such that
				\begin{equation}{\label{supercritical case initial data2}}
					\begin{split}
						M^*_1M^*_2/(M^{*m_c}_1+M^{*m_c}_2)=M^{2/d}_c/2,
					\end{split}
				\end{equation}
				there exists nonnegative function $(u^*,w^*)$ with $\|u^*\|_1=M^*_1$,$\|w^*\|_1=M^*_2$ fulfilling $\mathcal{F}[u^*,w^*]=0$.

				This can be seen by the fact that $C_c$ in (\ref{upper for critical number}) is
				\begin{equation*}
					\begin{split}
						C_{c}=\sup_{f\neq 0}\left\{\frac{\mathcal{H}[f,f]}{\|f\|^{2/d}_{1}\|f\|^{m_c}_{m_c}},\,\,\,f\in L^1(\mathbb{R}^d)\cap L^{m_c}(\mathbb{R}^d)\right\}
					\end{split}
				\end{equation*}
				from Theorem {\ref{extremals of H in case 3}}. From {\cite[Proposition 3.3]{Carrillo09-CVPDE}}, for any $M^*_1>0$ there exists nonnegative, radially symmetric and non-increasing function $u^*\in L^1(\mathbb{R}^d)\cap L^{m_c}(\mathbb{R}^d)$ \text{with}\,\,\,$\|u^*\|_1=M^*_1$ such that
				\begin{equation}{\label{supreme for u}}
					\begin{split}
						\|u^*\|^{m_c}_{m_c}=C^{-1}_c\|u^*\|^{-2/d}_1\mathcal{H}[u^*,u^*].
					\end{split}
				\end{equation}
				Define $w^*=M^*_2/M^*_1u^*$. Then $w^*\in L^1(\mathbb{R}^d)\cap L^{m_c}(\mathbb{R}^d)$ \text{with}\,\,\,$\|w^*\|_1=M^*_2$ and
				\begin{equation*}
					\begin{split}
						\mathcal{F}[u^*,w^*]=&0
					\end{split}
				\end{equation*}
				by (\ref{supercritical case initial data2}) and the definition of $M_c$. Then
				\begin{equation*}
					\begin{split}
						 c_d\mathcal{H}[u^*,w^*]=c_dM^*_2/M^*_1\mathcal{H}[u^*,u^*]=\frac{1}{m_c-1}\left(1+\left(\frac{M^*_2}{M^*_1}\right)^{m_c}\right)\|u^*\|^{m_c}_{m_c}.
					\end{split}
				\end{equation*}
				Given $u_0=\frac{M_1}{M^*_1}u^*$ and $w_0=\frac{M_2}{M^*_2}w^*$ with $\|u_0\|_1=M_1$ and $\|w_0\|_1=M_2$, then
				\begin{equation*}
					\begin{split}
						\mathcal{F}[u_0,w_0]=&\frac{1}{{m_c}-1}\|u_0\|^{m_c}_{m_c}+\frac{1}{{m_c}-1}\|w_0\|^{m_c}_{m_c}-c_d\mathcal{H}[u_0,w_0]\\
						=&\frac{1}{{m_c}-1}\left[\left(\frac{M_1}{M^*_1}\right)^{m_c}+\left(\frac{M_2}{M^*_1}\right)^{m_c}
						 -\frac{M_1M_2}{M^{*}_1M^{*}_2}\left(1+\left(\frac{M^*_2}{M^*_1}\right)^{m_c}\right)\right]\|u^*\|^{m_c}_{m_c}\\
						<& 0,
					\end{split}
				\end{equation*}
				since
				\begin{equation*}
					\begin{split}
						M_1M_2/(M^{{m_c}}_1+M^{{m_c}}_2)>M^*_1M^*_2/(M^{*{m_c}}_1+M^{*{m_c}}_2)=M^{2/d}_c/2.
					\end{split}
				\end{equation*}
				If $(u,w)$ is corresponding free energy solution with the initial data $(u_0,w_0)$, then
				\begin{equation*}
					\begin{split}
						\mathcal{F}[u(t),w(t)]\leq \mathcal{F}[u_0,w_0]<0,\,\,\,\,t>0
					\end{split}
				\end{equation*}
				by the decreasing property of $\mathcal{F}$. From Lemma \ref{second moment}, it follows that blow up occurs.

				To see the simultaneous blow-up phenomenon, from extensibility criterion in Theorem \ref{local existence theorem} we have
				\begin{equation*}\label{comparison between um1 and wm21}
					\begin{split}
						C\|w(t)\|^{m_2}_{m_2}+C\leq \|u(t)\|^{m_1}_{m_1}\leq C'\|w(t)\|^{m_2}_{m_2}+C'\,\,\,\,\text{for}\,\,\,\,t\in(0,T_{\max})
					\end{split}
				\end{equation*}
				with some $C>0$	and $C'>0$ if $\bf{m}$ is critical. Then all assertions have been proved.
			\end{proof}

			{\bf Acknowledgment.} JAC was supported the Advanced Grant Nonlocal-CPD (Nonlocal PDEs for Complex Particle Dynamics: 	 Phase Transitions, Patterns and Synchronization) of the European Research Council Executive Agency (ERC) under the European Union's Horizon 2020 research and innovation programme (grant agreement No. 883363). JAC was also partially supported by the EPSRC grant number EP/P031587/1. JAC acknowledges support through the Changjiang Visiting Professorship Scheme of the Chinese Ministry of Education. KL is partially supported by NSFC (Grant No. 11601516) and by Sichuan Science and Technology Program (Grant No. 2020YJ0060).

		\end{document}